\documentclass[english]{article}
\usepackage[T1]{fontenc}
\usepackage[latin9]{inputenc}
\usepackage{color}
\usepackage{refstyle}
\usepackage{mathrsfs}
\usepackage{enumitem}
\usepackage{amsmath}
\usepackage{amsthm}
\usepackage{amssymb}
\usepackage{setspace}
\makeatletter

\AtBeginDocument{\providecommand\secref[1]{\ref{sec:#1}}}
\AtBeginDocument{\providecommand\subsecref[1]{\ref{subsec:#1}}}
\AtBeginDocument{\providecommand\lemref[1]{\ref{lem:#1}}}
\AtBeginDocument{\providecommand\thmref[1]{\ref{thm:#1}}}
\AtBeginDocument{\providecommand\corref[1]{\ref{cor:#1}}}
\AtBeginDocument{\providecommand\propref[1]{\ref{prop:#1}}}
\RS@ifundefined{subsecref}
  {\newref{subsec}{name = \RSsectxt}}
  {}
\RS@ifundefined{thmref}
  {\def\RSthmtxt{theorem~}\newref{thm}{name = \RSthmtxt}}
  {}
\RS@ifundefined{lemref}
  {\def\RSlemtxt{lemma~}\newref{lem}{name = \RSlemtxt}}
  {}

\theoremstyle{plain}
\newtheorem{thm}{\protect\theoremname}[section]
\theoremstyle{remark}
\newtheorem{rem}[thm]{\protect\remarkname}
\theoremstyle{definition}
\newtheorem{defn}[thm]{\protect\definitionname}
\theoremstyle{plain}
\newtheorem{lem}[thm]{\protect\lemmaname}
\theoremstyle{plain}
\newtheorem{prop}[thm]{\protect\propositionname}
\theoremstyle{plain}
\newtheorem{cor}[thm]{\protect\corollaryname}
\theoremstyle{definition}
\newtheorem{example}[thm]{\protect\examplename}

\@ifundefined{date}{}{\date{}}
\AtBeginDocument{}
\AtBeginDocument{}
\AtBeginDocument{\providecommand\propref[1]{\ref{prop:#1}}}
\RS@ifundefined{propref}{\newref{prop}{name =Proposition~,names = Propositions~}}{}
\AtBeginDocument{\providecommand\corref[1]{\ref{cor:#1}}}
\RS@ifundefined{corref}{\newref{cor}{name =Corollary~,names = Corollaries~}}{}
\RS@ifundefined{remref}{\newref{rem}{name =Remark~,names = Remarks~}}{}
\RS@ifundefined{defref}{\newref{def}{name =Definition~,names = Definitions~}}{}
\RS@ifundefined{exaref}{\newref{exa}{name =Example~,names = Examples~}}{}
\RS@ifundefined{eqqref}{\newref{eqq}{name =Equality~,names = Equality~}}{}
\AtBeginDocument{\renewcommand\subsecref[1]{Section~\ref{subsec:#1}}}
\AtBeginDocument{\renewcommand\secref[1]{Section~\ref{sec:#1}}}

\def\@fnsymbol#1{\ensuremath{\ifcase#1\or *\or **\or \ddagger\or
   \mathsection\or \mathparagraph\or \|\or \dagger\dagger
   \or \ddagger\ddagger \else\@ctrerr\fi}}

\usepackage{enumitem, hyperref}
\makeatletter
\def\namedlabel#1#2{\begingroup
    #2%
    \def\@currentlabel{#2}%
    \phantomsection\label{#1}\endgroup
}

\usepackage{tikz}

\usepackage{tocbibind}
\usepackage{youngtab}
\usepackage{authblk}
\usetikzlibrary{matrix}
\usepackage{bbding}
\usepackage{ytableau}
\usepackage{amsmath}
\usepackage{geometry}
\usepackage{mathtools}
\usepackage{mathrsfs}

\tikzset{   pt/.style={insert path={node[scale=2]{.}}},   dnup/.style={insert path={ [pt] .. controls +(0,1) and +(0,-1) .. +(#1,2) [pt]}},   dndn/.style={insert path={ [pt] .. controls +(0,0.25) and +(0,0.25) .. +(#1,0) [pt]}},   upup/.style={insert path={ [pt] .. controls +(0,-0.25) and +(0,-0.25) .. +(#1,0) [pt]}}, upup2/.style={insert path={ [pt] .. controls +(0,-0.5) and +(0,-0.5) .. +(#1,0) [pt]}}, }

\newcommand{\Lt}{\widetilde{\mathcal{L}}_E}
\newcommand{\Ht}{\widetilde{\mathcal{H}}_E}

\DeclareMathOperator{\Hom}{Hom}

\DeclareMathOperator{\id}{\bf{id}}
\DeclareMathOperator{\h}{\bf{h}}

\DeclareMathOperator{\sgn}{sgn}
\DeclareMathOperator{\Rad}{Rad}
\DeclareMathOperator{\im}{\mathsf{im}}

\DeclareMathOperator{\Span}{span}

\DeclareMathOperator{\Rc}{\mathcal{R}}
\DeclareMathOperator{\Lc}{\mathcal{L}}
\DeclareMathOperator{\Hc}{\mathcal{H}}

\DeclareMathOperator{\Jc}{\mathcal{J}}

\DeclareMathOperator{\op}{op}

\DeclareMathOperator{\db}{\mathbf{d}}
\DeclareMathOperator{\rb}{\mathbf{r}}

\DeclareMathOperator{\VS}{\mathbf{VS}}

\DeclareMathOperator{\Op}{O}

\DeclareMathOperator{\OFD}{OFD}
\DeclareMathOperator{\OD}{OD}
\DeclareMathOperator{\COD}{COD}

\def\RSlemtxt{Lemma~}
\def\RSthmtxt{Theorem~}

\usepackage{babel}
\providecommand{\corollaryname}{Corollary}
  \providecommand{\definitionname}{Definition}
  \providecommand{\examplename}{Example}
  \providecommand{\lemmaname}{Lemma}
  \providecommand{\propositionname}{Proposition}
  \providecommand{\remarkname}{Remark}
 
\providecommand{\theoremname}{Theorem}
\makeatother

\usepackage{babel}
\providecommand{\corollaryname}{Corollary}
\providecommand{\definitionname}{Definition}
\providecommand{\examplename}{Example}
\providecommand{\lemmaname}{Lemma}
\providecommand{\propositionname}{Proposition}
\providecommand{\remarkname}{Remark}
\providecommand{\theoremname}{Theorem}

\begin{document}
\title{Representation theory of monoids consisting of order-preserving functions
and order-reversing functions on an $n$-set}
\author{Itamar Stein\thanks{Mathematics Unit, Shamoon College of Engineering, 77245 Ashdod, Israel}\\
\Envelope \, Steinita@gmail.com}
\maketitle
\begin{abstract}
We consider the monoid of order-preserving functions and order-reversing
functions on an $n$-set. Over a field of characteristic not \,2,
we obtain a quiver presentation for the monoid algebra: two straight
paths with $n-1$ and $n$ vertices, respectively, and all compositions
are zero. We also classify homomorphisms between induced left Schützenberger
modules. We define a covering monoid with a distinction between order-preserving
and order-reversing constant functions. We show that the algebra of
the covering monoid is isomorphic to the direct product of two copies
of the algebra of the monoid of all order-preserving functions on
an n-set.
\end{abstract}

\section{Introduction }

In the representation theory of finite monoids, it is common to study
algebras of monoids that possess a natural combinatorial structure.
In particular, a typical choice is to investigate algebras of transformation
monoids (see, for example, \cite{Grensing2014,Grood2002,Putcha1996,Putcha1998,Solomon2002,Stein2016,Stein2019,Stein2020,Stein2024preprint,Steinberg2006,Steinberg2008,Steinberg2016B};
see also \cite[Chapter 13]{Steinberg2016} and \cite[Chapter 11]{Ganyushkin2009b}).

In this paper, we study the algebras of two closely related monoids.
The monoid $\OD_{n}$, which consists of all order-preserving functions
and order-reversing functions on the set $[n]=\{1,\dots,n\}$ (see
\cite{Fernandes2005,Umar2014,Fernandes2015}). We also introduce a
monoid, denoted $\COD_{n}$, which resembles $\OD_{n}$ but with a
distinction between order-preserving and order-reversing constant
functions. For $\OD_{n}$ our goal is to find a quiver presentation
for its monoid algebra $\Bbbk\OD_{n}$ where $\Bbbk$ is a field whose
characteristic is not 2.

A quiver presentation of an algebra $A$ is a standard method for
describing an algebra in the theory of associative algebras \cite[Chapter II]{Assem2006},\cite[Chapter 3]{Barot2015}.
It consists of a (unique) generating graph $Q$, called the quiver
of $A$, and a (non-unique) set $R$ of relations between paths in
$Q$. Every split basic algebra $B$ can be presented by such tuple
$(Q,R)$. If $A$ is Morita equivalent to a split basic algebra $B$
(that is, the module categories of $A$ and $B$ are equivalent) then
a quiver presentation for $B$ is considered also as a quiver presentation
for $A$. 

If an algebra $A$ is defined by other means, then finding a quiver
presentation for it becomes a natural question. In particular, if
$A=\Bbbk M$ is a monoid algebra, where $M$ is a finite monoid and
$\Bbbk$ is a field, then determining its quiver or providing a quiver
presentation is a fundamental problem in the representation theory
of finite monoids \cite[Chapter 17]{Steinberg2016}. There are many
monoids and classes of monoids for which a description of the quiver
is known \cite{Denton2011,Margolis2015,Margolis2011,Margolis2012,Margolis2018B,Saliola2007,Shahzamanian2021,Stein2016}.
Relatively, far less is known about quiver presentations \cite{Saliola2023,Margolis2021B,Ringel2000,Saliola2009,Stein2020,Steinberg2025}.

In \secref{Preliminaries}, we provide the required preliminaries.
In particular, in \subsecref{Pre_Algebras-and-modules}, we associate
to each $\Lc$-class $L$ of $M$ an $\Bbbk M$-module $\Bbbk L$
known as the induced Schützenberger module of $L$. Then we turn to
examine the monoid $\OD_{n}$ in detail. In \subsecref{Elementary_observations},
we prove that all sandwich matrices of $\OD_{n}$ (to be defined in
\subsecref{Pre_Monoids}) are right invertible. By a result of Margolis
and Steinberg, the right invertibility of the sandwich matrices implies
that the regular left module $\Bbbk\OD_{n}$ is a direct sum of the
induced Schützenberger modules of its $\Lc$-classes. In particular,
this implies that $\Bbbk L$ is a projective module for every $\Lc$-class
$L$. In \subsecref{Construction-of-homomorphisms}, we describe in
detail the linear category formed by the induced Schützenberger modules
and homomorphisms between them. The opposite of this category is essentially
the same as the algebra $\Bbbk\OD_{n}$ in the sense that their module
categories are equivalent, allowing us to derive information about
the quiver presentation from this linear category. To obtain a split
basic algebra, we switch our attention in \subsecref{The-skeletal-category}
to the skeleton of this linear category and, in \subsecref{Composition-of-homomorphisms},
find relations between its morphisms. This provides sufficient information
about the associated algebra to compute its quiver presentation directly
by finding a complete set of primitive orthogonal idempotents. The
main theorem on $\Bbbk\OD_{n}$ states that its quiver consists of
two straightline paths, one with $n-1$ vertices and one with $n$
vertices, and that all compositions of consecutive arrows are equal
to $0$. 

In \secref{COD_section}, we turn to the monoid $\COD_{n}$ and define
it as a submonoid of $\OD_{n}\times\mathbb{Z}_{2}$. From an algebraic
point of view, we obtain more elegant results for $\COD_{n}$ and
its algebra. We prove that $\COD_{n}\simeq\Op_{n}\rtimes\mathbb{Z}_{2}$,
where $\Op_{n}$ is the monoid of all order-preserving functions on
$[n]$. Moreover, adapting the approach of \secref{Order_preserving},
we obtain that $\Bbbk\COD_{n}\simeq\Bbbk\Op_{n}\times\Bbbk\Op_{n}$,
where $\Bbbk$ is a field whose characteristic is not 2. Finally,
we show that the quiver of $\Bbbk\COD_{n}$ consists of two straightline
paths with $n$ vertices, and that all compositions of consecutive
arrows are equal to $0$.

\textbf{Acknowledgment}: The author is grateful to Professor Abdullahi
Umar for suggesting the investigation of the representation theory
of the monoid $\OD_{n}$, and to the referee for a careful reading
and for suggestions that improved the paper.

\section{\label{sec:Preliminaries}Preliminaries}

\subsection{\label{subsec:Pre_Monoids}Monoids}

Let $M$ be a monoid. Green's preorders $\leq_{\Rc}$, $\leq_{\Lc}$
and $\leq_{\Jc}$ are defined by:
\begin{align*}
a\,\leq_{\Rc}\,b\iff & aM\subseteq bM\\
a\,\leq_{\Lc}\,b\iff & Ma\subseteq Mb\\
a\,\leq_{\Jc}\,b\iff & MaM\subseteq MbM.
\end{align*}
The associated Green's equivalence relations on $M$ are denoted by
$\Rc$, $\Lc$ and $\Jc$. Recall also that $\Hc=\Rc\cap\Lc$. It
is well-known that $\Lc$ ($\Rc$) is a right congruence (respectively,
left congruence). An element $a\in M$ is called\emph{ regular} if
there exists $b\in M$ such that $aba=a$. The monoid $M$ is \emph{regular}
if all its elements are regular.

If $M$ is finite then $\Jc=\Rc\circ\Lc=\Lc\circ\Rc$. So we can view
a specific $\Jc$-class $J$ as a matrix (or an ``eggbox'') whose
rows are indexed by the $\Rc$-classes in $J$ and its columns are
indexed by the $\Lc$-classes in $J$. Every entry in this matrix
is an $\Hc$-class. A $\Jc$-class $J$ is called \emph{regular} if
it contains a regular element. In this case, every element of $J$
is regular, and every $\Rc$-class and every $\Lc$-class in $J$
contains an idempotent. An $\Hc$-class contains an idempotent if
and only if it is a (maximal) subgroup of $M$. All group $\Hc$-classes
in the same regular $\Jc$-class $J$ are isomorphic so we can denote
this group by $G_{J}$. Let $J$ be a regular $\Jc$-class whose $\Rc$-classes
are indexed by $I$ and its $\Lc$-classes are indexed by $\Lambda$
(so $J$ is a $I\times\Lambda$ matrix of $\Hc$-classes). If $R_{i}$
is an $\Rc$-class and $L_{\lambda}$ is an $\Lc$-class where $i\in I$
and $\lambda\in\Lambda$, then we denote the intersection $\Hc$-class
$R_{i}\cap L_{\lambda}$ by $H_{i,\lambda}$. Without loss of generality,
assume that $1\in I,\Lambda$ is an index in both sets such that $H_{1,1}\simeq G_{J}$
is a group $\Hc$-class. Moreover, for every $i\in I$ we choose a
representative $a_{i}\in H_{i,1}$ such that $a_{1}\in H_{1,1}\simeq G_{J}$
is the identity element of $G_{J}$ (i.e., the unique idempotent of
$H_{1,1}).$ For every $\lambda\in\Lambda$ we choose a representative
$b_{\lambda}\in H_{1,\lambda}$ such that $b_{1}=a_{1}$. It is known
that $b_{\lambda}a_{i}\in H_{1,1}$ if and only if $H_{i,\lambda}$
contains an idempotent (i.e., it is a group $\Hc$-class). The \emph{sandwich
matrix $P_{J}$} associated with $J$ is a $\Lambda\times I$ matrix
over $G_{J}\cup\{0\}$ defined by 
\[
[P_{J}]_{\lambda,i}=\begin{cases}
b_{\lambda}a_{i} & \text{\ensuremath{H_{i,\lambda}} contains an idempotent}\\
0 & \text{otherwise}
\end{cases}
\]
where $\lambda\in\Lambda$ and $i\in I$. We will omit $J$ and write
$P$ if the $\Jc$-class is clear from context. It is known that the
invertibility of $P_{J}$ over $\Bbbk G_{J}$ does not depend on the
choice of representatives. The reason behind this definition and additional
basic notions and proofs of monoid theory can be found in standard
textbooks such as \cite{Howie1995}.

\subsection{Categories and linear categories}

Let $Q$ be a (finite, directed) graph. We denote by $Q^{0}$ its
set of vertices (or \emph{objects}) and by $Q^{1}$ its set of edges
(or \emph{morphisms}). We allow $Q$ to have more than one morphisms
between two objects. We denote by $\db$ and $\rb$ the domain and
range functions
\[
\db,\rb:Q^{1}\to Q^{0}
\]
assigning to every morphism $m\in Q^{1}$ its \emph{domain $\db(m)$
}and its \emph{range $\rb(m)$.} The set of edges with domain $a$
and range $b$ (also called an \emph{hom-set}) is denoted $Q(a,b)$.

Since any category has an underlying graph, the above notations hold
also for categories. For every graph $Q$, denote by $Q^{\ast}$ the
\emph{free category} generated by the graph $Q$. The object sets
of $Q^{\ast}$ and $Q$ are identical but the morphisms of $Q^{\ast}$
are the paths in $Q$ (including one empty path for every object).
The composition in $Q^{\ast}$ is concatenation of paths (from right
to left) and the empty paths are the identity elements.

A category $E$ is called a \emph{category with zero morphisms} if
for every $x,y\in E^{0}$ there exists a zero morphism $z_{y,x}\in E(x,y)$
and for every three objects $x,y,w\in E^{0}$ and morphisms $m^{\prime}\in E(y,w)$,
$m\in E(x,y)$ we have 
\[
m^{\prime}z_{y,x}=z_{w,x}=z_{w,y}m.
\]

A \emph{relation} $R$ on a category $E$ is a relation on the set
of morphisms $E^{1}$ such that $mRm^{\prime}$ implies that $\db(m)=\db(m^{\prime})$
and $\rb(m)=\rb(m^{\prime})$. A relation $\theta$ on a category
$E$ is called a (category) congruence if $\theta$ is an equivalence
relation with the property that $m_{1}\theta m_{2}$ and $m_{1}^{\prime}\theta m_{2}^{\prime}$
implies $m_{1}^{\prime}m_{1}\theta m_{2}^{\prime}m_{2}$ whenever
$\rb(m_{1})=\db(m_{1}^{\prime})$ (and $\rb(m_{2})=\db(m_{2}^{\prime})$).

Let $\Bbbk$ be a field. A \emph{$\mathbb{\Bbbk}$-linear category}
is a category $L$ enriched over the category of $\mathbb{\Bbbk}$-vector
spaces $\VS_{\mathbb{\Bbbk}}$. This means that every hom-set of $L$
is a $\mathbb{\Bbbk}$-vector space and the composition of morphisms
is a bilinear map with respect to the vector space operations. We
will discuss only $\Bbbk$-linear categories where the hom-sets are
finite-dimensional $\Bbbk$-vector spaces.

For any category $E$, we associate a $\mathbb{\Bbbk}$-linear category
$L_{\Bbbk}[E]$, called the \emph{linearization} of $E$, defined
in the following way. The objects of $L_{\Bbbk}[E]$ and $E$ are
identical, and every hom-set $L_{\Bbbk}[E](a,b)$ is the $\Bbbk$-vector
space with basis $E(a,b)$. The composition of morphisms in $L_{\Bbbk}[E]$
is defined naturally in the only way that extends the composition
of $E$ and forms a bilinear map.

Let $L$ be a $\mathbb{\Bbbk}$-linear category. A relation $\rho$
on $L^{1}$ is called a ($\mathbb{\Bbbk}$-linear category) \emph{congruence}
if $\rho$ is a category congruence and also a vector space congruence
on every hom-set $L(a,b)$. By a vector space congruence on a vector
space $V$ we mean an equivalence relation $\rho$ on $V$ such that
$u_{1}\rho v_{1}$ and $u_{2}\rho v_{2}$ implies $(u_{1}+u_{2})\rho(v_{1}+v_{2})$
and $u\rho v$ implies $(ku)\rho(kv)$ for every $k\in\Bbbk$. The
quotient $\Bbbk$-linear category $L/\rho$ is then defined in the
natural way. For every relation $R$ on $L$ there exists a unique
minimal congruence on $L$ that contains $R$. We will denote this
congruence by $\rho_{R}$. Now we can define a presentation of $\mathbb{\Bbbk}$-linear
categories. Let $Q$ be a subgraph of $L$ such that $Q^{0}=L^{0}$.
The free $\mathbb{\Bbbk}$-linear category generated by $Q$ is the
category $L_{\Bbbk}[Q^{\ast}]$ of all linear combinations of paths.
If $R$ is a relation on $L_{\Bbbk}[Q^{\ast}]$ such that $L_{\Bbbk}[Q^{\ast}]/\rho_{R}\simeq L$,
then we say that $(Q,R)$ is a \emph{presentation} of $L$.

\subsection{\label{subsec:Pre_Algebras-and-modules}Algebras and modules}

A $\Bbbk$-\emph{algebra} is a $\Bbbk$-linear category with a single
object. Every $\Bbbk$-algebra $A$ considered in this paper is finite-dimensional
and unital, and all $A$-modules are assumed to be finite-dimensional. 

Let $L$ be a $\Bbbk$-linear category. There is a $\Bbbk$-algebra
associated with $L$ that we will denote $\mathcal{\mathscr{A}}[L]$.
As a set $\mathcal{\mathscr{A}}[L]$ consists of combinations of morphisms
form $L$:
\[
\mathcal{\mathscr{A}}[L]=\{x_{1}+\ldots+x_{n}\mid x_{i}\in L^{1}\}
\]
Addition and scalar multiplication are defined in the natural way.
Multiplication is defined as an extension of 
\[
y\cdot x=\begin{cases}
yx & \text{if \ensuremath{yx} is defined in }L\\
0 & \text{otherwise}
\end{cases}
\]
where $x,y\in L^{1}$.

We view the algebra $\mathscr{A}[L]$ as essentially the same object
as $L$. The difference is that in the algebra, we have ``forgotten''
the multiple objects, thereby losing some information that $L$ carries.
Conversely, we can start with an algebra $A$ that we wish to decompose
into a multi-object linear category. This can be done in the language
of projective modules or primitive idempotents, both will be useful
in this paper. 

Recall that an $A$-module $P$ is called \emph{projective} if the
functor $\Hom_{A}(P,-)$ is an exact functor. Equivalently, $P$ is
projective if it is a direct summand of $A^{k}$ (for some $k\in\mathbb{N}$)
as an $A$-module.

Let $A\simeq{\displaystyle \bigoplus_{i=1}^{n}P_{i}}$ be a decomposition
of a $\Bbbk$-algebra $A$ into a direct sum of projective modules.
Let $\mathcal{L}$ be the $\Bbbk$-linear category whose objects are
the modules $\{P_{1},\ldots,P_{n}\}$ and the morphisms are all the
homomorphisms between them. Then, it is well-known that $\mathscr{A}[\mathcal{L}]\simeq A^{\op}$
as $\Bbbk$-algebras. We reach a maximum number of objects in $\mathcal{L}$
if we decompose $A$ into indecomposable projective modules.

For the quiver presentation, it will be convenient to describe the
decomposition of $A$ into a linear category also in the language
of primitive idempotents. \textcolor{black}{Recall that two idempotents
}$e,f\in A$\textcolor{black}{{} are called orthogonal if }$ef=fe=0$.\textcolor{black}{{}
A non-zero idempotent} $e\in A$ is called \emph{primitive} if it
is not a sum of two non-zero orthogonal idempotents. This is equivalent
to $eAe$ being a local algebra (i.e., an algebra with no non-trivial
idempotents). It is known that $e\in A$ is primitive if and only
if $Ae$ is an indecomposable projective module. A \emph{complete
set of primitive orthogonal idempotents }is a set of primitive, mutually
orthogonal idempotents $\{e_{1},\ldots,e_{r}\}$ whose sum is $1$.
We can associate to every algebra $A$ a linear category, denoted
$\mathscr{L}(A)$, in the following way. The objects are in one-to-one
correspondence with a complete set of primitive idempotents. The hom-set
$\mathscr{L}(A)(e_{i},e_{j})$ is the set $e_{j}Ae_{i}$, that is,
all the elements $a\in A$ such that $e_{j}ae_{i}=a$. Composition
of two morphisms is naturally defined as their product in the algebra
$A$. The algebra $\mathscr{A}[\mathscr{L}(A)]$ is isomorphic to
$A$.

We now add the assumption that $A$ is split basic. We denote by $\Rad A$
the Jacobson radical of $A$, which is the minimal ideal $I$ such
that $A/I$ is a semisimple algebra. The algebra $A$ is called \emph{split
basic} if $A/\Rad A\simeq\Bbbk^{n}$, i.e, the maximal semisimple
quotient of $A$ is a direct product of the base field.\textcolor{red}{{}
}The \emph{(ordinary) quiver} of a split basic algebra $A$ is a directed
graph $Q$ defined as follows: The set of vertices of $Q$ is in one-to-one
correspondence with a complete set of primitive orthogonal idempotents
$\{e_{1},\ldots,e_{n}\}$. The edges (more often called \emph{arrows})
from $e_{k}$ to $e_{r}$ are in one-to-one correspondence with a
basis of the vector space $(e_{r}+\Rad^{2}A)(\Rad A/\Rad^{2}A)(e_{k}+\Rad^{2}A)$.
It is well-known that this definition does not depend on the specific
choice of the primitive orthogonal idempotents. The quiver $Q$ of
$A$ is naturally identified with a subgraph of $\mathscr{L}(A)$.
There exists a $\Bbbk$-linear relation $R$ on $L_{\Bbbk}[Q^{\ast}]$
such that $(Q,R)$ is a presentation for $\mathscr{L}(A)$ (\cite[Theorem 1.9 on page 65]{Auslander1997}).
Such a pair $(Q,R)$ is called a \emph{quiver presentation }for the
algebra $A$.

Two algebras are Morita equivalent if their module categories are
equivalent. For two split basic algebras, $A$ and $B$, that are
Morita equivalent, a quiver presentation for $A$ is also a quiver
presentation for $B$. Even if $A$ is not split basic but it is Morita
equivalent to a split basic algebra $B$, a quiver presentation $(Q,R)$
for $B$ is also regarded as a quiver presentation for $A$. This
is because the category of representations of the bound quiver $(Q,R)$
is equivalent to the category of $A$-modules.

Let $E$ be a finite category and let $\Bbbk$ be a field. The \emph{category
algebra $\Bbbk E$ }is defined as $\mathscr{A}[L_{\Bbbk}[E]]$. Equivalently,
we can say that $\Bbbk E$ consists of all formal linear combinations
\[
\{k_{1}m_{1}+\ldots+k_{n}m_{n}\mid k_{i}\in\mathbb{\Bbbk},\,m_{i}\in E^{1}\}
\]
and the multiplication in $\mathbb{\Bbbk}E$ is the linear extension
of the following:
\[
m^{\prime}\cdot m=\begin{cases}
m^{\prime}m & \text{if \ensuremath{m^{\prime}m} is defined}\\
0 & \text{otherwise}.
\end{cases}
\]

The unit element of $\Bbbk E$ is${\displaystyle \sum_{x\in E^{0}}1_{x}}$
where $1_{x}$ is the identity morphism of the object $x$ of $E$.

If $E$ is a category with zero morphisms, then the zero morphisms
span an ideal $Z$ of $\Bbbk E$. The \emph{contracted category algebra
}$\Bbbk_{0}E$ is defined as $\Bbbk E/Z$. It identifies all the zero
morphisms of $E$ with the zero of the algebra $\Bbbk E$.

Every finite monoid $M$ is a finite category with a single object.
Consequently, the monoid algebra $\Bbbk M$ is a special case of a
category algebra.

For every $\Lc$-class $L$ we denote by $\Bbbk L$ the $\Bbbk$-vector
space spanned by elements of $L$. It is known \cite[Page 57]{Steinberg2016}
that $\Bbbk L$ is a $\Bbbk M$-module whose operation is defined
by 
\[
s\bullet x=\begin{cases}
sx & sx\in L\\
0 & \text{otherwise}.
\end{cases}
\]

The module $\Bbbk L$ is also called the \emph{induced Schützenberger
module} of $L$.

The next theorem follows from \cite[Theorem 4.7 and its proof]{Margolis2018B}
when we specialize to the special case of a regular monoids.
\begin{thm}
\label{thm:Decomposition_Projective_modules}Let $M$ be a finite
regular monoid and let $\Bbbk$ be a field. Let $\{L_{1},\ldots,L_{n}\}$
be the $\Lc$-classes of $M$. Assume that for every $\Jc$ class
$J$ the sandwich matrix $P_{J}$ is right invertible over $\Bbbk G_{J}$.
Then, for every $i\in\{1,\ldots,n\}$ the module $\Bbbk L_{i}$ is
a projective $\Bbbk M$-module and 
\[
\Bbbk M\simeq\bigoplus_{i=1}^{n}\Bbbk L_{i}
\]
is a decomposition of $\Bbbk M$ into a sum of projective modules.
\end{thm}

\begin{rem}
In fact, \cite[Theorem 4.7]{Margolis2018B} tells us how to decompose
$\Bbbk L_{i}$ further to obtain indecomposable projective modules.
However, at the moment we will be content with this incomplete decomposition.
Later, we will handle the final phase of decomposition into indecomposable
projective modules in detail in order to find a quiver presentation.
\end{rem}

\section{\label{sec:Order_preserving}The monoid of order-preserving functions
and order-reversing functions}

\subsection{\label{subsec:Elementary_observations}Elementary observations}

For every $n\in\mathbb{N}$ we set $[n]=\{1,\ldots,n\}$. 
\begin{defn}
A function $f:[n]\to[n]$ is called \emph{order-preserving} if $x_{1}\leq x_{2}\implies f(x_{1})\le f(x_{2})$
and \emph{order-reversing} if $x_{1}\leq x_{2}\implies f(x_{1})\geq f(x_{2})$. 
\end{defn}

We denote by $\Op_{n}$ the monoid of all order-preserving functions
$f:[n]\to[n]$ and by $\OD_{n}$ the monoid of all functions $f:[n]\to[n]$
that are either order-preserving or order-reversing (for the latter
see \cite{Fernandes2005,Umar2014,Fernandes2015}).

Our main goal in this section is a description of a quiver presentation
for the algebra $\Bbbk\OD_{n}$ where $\Bbbk$ is a field whose characteristic
is not $2$. We start with elementary observations on $\OD_{n}$.

We denote by $P([n])$ the power set of $[n]$ and define $P_{1}([n])=\{X\in P([n])\mid1\in X\}$.
We denote by $\im(f)$ the image of a function $f:[n]\to[n]$. Recall
that the kernel of a function $f\in\OD_{n}$, denoted $\ker(f)$,
is the equivalence relation on its domain defined by $(x_{1},x_{2})\in\ker(f)$,
if and only if $f(x_{1})=f(x_{2})$. If $f$ is order-preserving or
order-reversing, the kernel classes of $f$ are \emph{intervals}:
if $x_{1}\leq x_{2}\leq x_{3}$ and $(x_{1},x_{3})\in\ker(f)$, then
$(x_{1},x_{2})\in\ker f$ as well. Let $f\in\OD_{n}$ and let $K_{1},\ldots,K_{l}$
be the kernel classes of $f$ arranged in an increasing order (if
$i\leq j$ and $x\in K_{i}$, $x^{\prime}\in K_{j}$, then $x\leq x^{\prime}$).
Choose $x_{i}$ to be the minimal element of $K_{i}$ (note that it
must be the case that $x_{1}=1$) and set $X=\{x_{1},\ldots,x_{l}\}$.
We call $X$ the \emph{kernel set} of $f$. Note also that if $X$
is the kernel set of $f$, then $|\im(f)|=|X|$.

From now on any list of elements $X=\{x_{1},\ldots,x_{l}\}\in P([n])$
will be ordered in an increasing order unless stated otherwise. For
every two sets $X\in P_{1}([n])$ and $Y\in P([n])$ such that $|X|=|Y|$
we can associate a unique order-preserving function $f_{Y,X}^{{\bf 1}}\in\OD_{n}$
whose kernel set is $X$ and image is $Y$. If $X=\{x_{1},\ldots,x_{l}\}$
and $Y=\{y_{1},\ldots,y_{l}\}$ then
\[
f_{Y,X}^{{\bf 1}}(x)=\begin{cases}
y_{i} & x_{i}\leq x<x_{i+1}\quad(1\leq i\leq l-1)\\
y_{l} & x_{l}\leq x.
\end{cases}
\]

The order-reversing functions of $\OD_{n}$ are also in one-to-one
correspondence with pairs of sets $X\in P_{1}([n])$ and $Y\in P([n])$
with the property that $|X|=|Y|$. We can associate a unique order-reversing
function $f_{Y,X}^{{\bf {\bf -1}}}\in\OD_{n}$ whose kernel set is
$X$ and image is $Y$:
\[
f_{Y,X}^{{\bf -1}}=\begin{cases}
y_{l+1-i} & x_{i}\leq x<x_{i+1}\quad(1\leq i\leq l-1)\\
y_{1} & x_{l}\leq x
\end{cases}
\]
A general element of $\OD_{n}$ may be written as $f_{Y,X}^{\alpha}$
where $\alpha\in\{{\bf 1},{\bf -1}\}$. The superscripts will be useful
because we will be able to write equalities like $f_{Z,Y}^{\beta}f_{Y,X}^{\alpha}=f_{Z,X}^{\beta\alpha}$
where $\beta\alpha$ is the product in the two elements group $\mathbb{Z}_{2}\simeq\{{\bf 1},{\bf -1}\}$.
We are not going to use inverse functions explicitly so $f_{Y,X}^{{\bf -1}}$
will cause no ambiguity. We also remark that $f_{Y,X}^{{\bf 1}}=f_{Y,X}^{-{\bf 1}}$
if and only if $|X|=|Y|=1$, and therefore $|\OD_{n}|=2|\Op_{n}|-n$
(see \cite[Proposition 1.3]{Fernandes2005}).

For every $X\in P_{1}([n])$ the function $f_{X,X}^{{\bf {\bf 1}}}$
is an idempotent that can be described by 
\[
f_{X,X}^{{\bf 1}}(i)=\max\{x\in X\mid x\leq i\}.
\]
It will be convenient to define $e_{X}=f_{X,X}^{{\bf 1}}$ where $X\in P_{1}([n])$.
We will use frequently the fact that $e_{Z}f_{Y,X}^{{\bf \alpha}}=f_{Y,X}^{{\bf \alpha}}$
if $Y\subseteq Z$ and $f_{Y,X}^{{\bf \alpha}}e_{Z}=f_{Y,X}^{{\bf \alpha}}$
if $X\subseteq Z$.

It is known that $\OD_{n}$ is a regular monoid and its Green's equivalence
relations are well-known (see \cite{Fernandes2005}). If $f_{Y_{1},X_{1}}^{\alpha},f_{Y_{2},X_{2}}^{\beta}\in\OD_{n}$
then
\begin{align*}
f_{Y_{1},X_{1}}^{\alpha}\,\Rc\,f_{Y_{2},X_{2}}^{\beta} & \iff\im(f_{Y_{1},X_{1}}^{\alpha})=\im(f_{Y_{2},X_{2}}^{\beta})\iff Y_{1}=Y_{2}\\
f_{Y_{1},X_{1}}^{\alpha}\,\Lc\,f_{Y_{2},X_{2}}^{\beta} & \iff\ker(f_{Y_{1},X_{1}}^{\alpha})=\ker(f_{Y_{2},X_{2}}^{\beta})\iff X_{1}=X_{2}\\
f_{Y_{1},X_{1}}^{\alpha}\,\Jc\,f_{Y_{2},X_{2}}^{\beta} & \iff\left|\im(f_{X_{1},Y_{1}})\right|=\left|\im(f_{X_{1},Y_{1}})\right|\iff|X_{1}|=|Y_{1}|=|X_{2}|=|Y_{2}|\\
f_{Y_{1},X_{1}}^{\alpha}\,\Hc\,f_{Y_{2},X_{2}}^{\beta} & \iff X_{1}=X_{2},\quad Y_{1}=Y_{2}.
\end{align*}

We want to show that all the sandwich matrices of the $\Jc$-classes
in $\OD_{n}$ are right invertible. For this we need a more detailed
description of the $\Jc$-classes. Consider the $\Jc$-class $J_{k}$
of all elements $f_{Y,X}^{\alpha}\in\OD_{n}$ with $|Y|=|X|=k$. The
$\Lc$-classes of $J_{k}$ are indexed by subsets $X\in P_{1}([n])$
such that $|X|=k$ so there are ${n-1 \choose k-1}$ $\Lc$-classes.
For every $X\in P_{1}([n])$ we denote by $L_{X}$ the associated
$\Lc$-class: $L_{X}=\{f_{Y,X}^{\alpha}\mid Y\in P([n]),\enspace|Y|=k\}$.
The $\Rc$-classes of $J_{k}$ are indexed by subsets $Y\in P([n])$
such that $|Y|=k$ so there are ${n \choose k}$ $\Rc$-classes. For
every $Y\in P([n])$ we denote by $R_{Y}$ the associated $\Rc$-class
$R_{Y}=\{f_{Y,X}^{\alpha}\mid X\in P_{1}([n]),\enspace|X|=k\}$. For
$X\in P_{1}([n])$ and $Y\in P([n])$ with $|X|=|Y|=k$ the associated
$\Hc$-class is $H_{Y,X}=R_{Y}\cap L_{X}=\{f_{Y,X}^{{\bf 1}},f_{Y,X}^{-{\bf 1}}\}$.
Note that 
\[
|H_{Y,X}|=\begin{cases}
2 & k>1\\
1 & k=1
\end{cases}
\]
so the group $G_{J}$ is $\mathbb{Z}_{2}$ if $k>1$ and trivial if
$k=1$. Denote by $\OFD_{n}$ the set of all functions $f_{Y,X}^{\alpha}\in\OD_{n}$
with $\im(f_{Y,X}^{\alpha})=Y\in P_{1}([n])$. Note that ${n-1 \choose k-1}$
of the $\Rc$-classes contain elements from $\OFD_{n}$ and ${n \choose k}-{n-1 \choose k-1}={n-1 \choose k}$
$\Rc$-classes contain elements that are not in $\OFD_{n}$. 
\begin{lem}
\label{lem:Sandwitch_right_invertible}The sandwich matrix $P$ of
every $\Jc$-class $J_{k}$ of $\OD_{n}$ is right invertible over
$\Bbbk G_{J}$.
\end{lem}

\begin{proof}
Arrange the $\Rc$-classes of $J_{k}$ such that the first ${n-1 \choose k-1}$
classes contain elements of $\OFD_{n}$ so the ``upper'' part of
$J_{k}$ is a ${n-1 \choose k-1}\times{n-1 \choose k-1}$ block with
the elements of $\OFD_{n}\cap J_{k}$. Let 
\[
X=\{1,x_{2},\ldots,x_{k}\},\quad Y=\{1,y_{2},\ldots,y_{k}\}
\]
be two subsets of $[n]$ of size $k$. We write $X\leq Y$ if $x_{i}\leq y_{i}$
for every $i$. Let $\preceq$ be a total order on the set $\{X\in P_{1}([n])\mid|X|=k\}$
such that $X\leq Y\implies X\preceq Y$. Order the $\Lc$-classes
and the first ${n-1 \choose k-1}$ $\Rc$-classes of $J_{k}$ according
to $\preceq$. If $f_{Y,X}^{\alpha}$ is an idempotent for $X,Y\in P_{1}([n])$,
then it must be the case that $\alpha={\bf 1}$ and $x_{i}\leq y_{i}$
for every $i$, and hence $X\leq Y\implies X\preceq Y$. So all the
idempotents in $J_{k}$ are ``below'' or on the diagonal of the
${n-1 \choose k-1}\times{n-1 \choose k-1}$ block. We choose representatives
as follows. Set $Z=\{1,\ldots,k\}\in P_{1}([n])$. Choose $a_{Y}=f_{Y,Z}^{{\bf 1}}\in R_{Y}$
for every $Y\in P([n])$ with $|Y|=k$, and $b_{X}=f_{Z,X}^{{\bf 1}}\in L_{X}$
for every $X\in P_{1}([n])$ with $|X|=k$. Now consider the sandwich
matrix $P$ of $J$. If $Y\prec X$, then $H_{Y,X}$ is ``above''
the diagonal, so it does not contain an idempotent. Therefore, $P_{X,Y}=0$.
Note that every ``diagonal'' $\Hc$-class $H_{X,X}$ contains the
idempotent $e_{X}=f_{X,X}^{{\bf 1}}$. Therefore, $P_{X,X}=b_{X}a_{X}=f_{Z,Z}^{{\bf 1}}=e_{Z}$
which is also an idempotent. Hence, it is the unit of $G_{J}$, and
we can write $P_{X,X}=1_{G}$. Therefore, $P$ is a ${n-1 \choose k-1}\times{n \choose k}$
matrix such that $P_{X,X}=1_{G}$ and $P_{X,Y}=0$ if $Y\prec X$
(where $X,Y\in P_{1}([n])$ such that $|X|=|Y|=k$). Therefore, the
left ${n-1 \choose k-1}\times{n-1 \choose k-1}$ block of $P$ is
an upper unitriangular matrix and therefore an invertible matrix over
$\Bbbk G_{J}$. This implies that $P$ is right invertible over $\Bbbk G_{J}$.
\end{proof}
\lemref{Sandwitch_right_invertible} and \thmref{Decomposition_Projective_modules}
now imply:
\begin{prop}
\label{prop:Peirce_decomposition}Let $\Bbbk$ be a field, then for
every $X\in P_{1}([n])$ the $\Bbbk\OD_{n}$-module $\Bbbk L_{X}$
is a projective module and 
\[
\Bbbk\OD_{n}\simeq\bigoplus_{X\in P_{1}([n])}\Bbbk L_{X}
\]
is an isomorphism of $\Bbbk\OD_{n}$-modules. 
\end{prop}

\subsection{\label{subsec:Construction-of-homomorphisms}Construction of homomorphisms}

For every $X,Y\in P_{1}([n])$ define $\widetilde{H}_{Y,X}=\{f_{W,X}^{\alpha}\in\OD_{n}\mid W\cup\{1\}=Y\}$.
Note that 
\[
\widetilde{H}_{Y,X}=\begin{cases}
\{f_{Y,X}^{{\bf 1}},f_{Y,X}^{{\bf -1}}\}=H_{Y,X} & |Y|=|X|\\
\{f_{Y\backslash\{1\},X}^{{\bf 1}},f_{Y\backslash\{1\},X}^{{\bf -1}}\}=H_{Y\backslash\{1\},X} & |Y|=|X|+1\\
\varnothing & |Y|\notin\left\{ |X|,|X|+1\right\} .
\end{cases}
\]
Every element $f_{W,Z}^{\alpha}$ is in precisely one set of the form
$\widetilde{H}_{Y,X}$ (namely, $X=Z$ and $Y=W\cup\{1\}$) so 
\[
{\displaystyle \sum_{X,Y\in P_{1}([n])}|\widetilde{H}_{Y,X}|}=|\OD_{n}|.
\]

\begin{rem}
If $\widetilde{H}_{Y,X}\neq\varnothing$ then it is an $\Ht$-class,
in the sense of \cite{Lawson1991}, where $E=\{e_{X}\mid X\in P_{1}([n])\}$.
\end{rem}

In this subsection we describe all homomorphisms between projective
modules of the form $\Bbbk L_{X}$ for $X\in P_{1}([n])$. In particular,
we show that $\dim_{\Bbbk}\Hom(\Bbbk L_{Y},\Bbbk L_{X})=|\widetilde{H}_{Y,X}|$
for every $X,Y\in P_{1}([n])$.

The easy case is to describe $\Hom(\Bbbk L_{Y},\Bbbk L_{X})$ where
$|X|=|Y|$. Recall that we denote by $\bullet$ the action of $\OD_{n}$
on $\Bbbk L$ where $L$ is an $\Lc$-class of $\OD_{n}$. Recall
also that $e_{X}=f_{X,X}^{{\bf 1}}$ for $X\in P_{1}([n])$. 
\begin{lem}
\label{lem:Rho_hom}Let $X,Y\in P_{1}([n])$ such that $|X|=|Y|$,
and let $\alpha\in\{{\bf 1},-{\bf 1}\}$. Then, the function $\mbox{\ensuremath{\rho_{Y,X}^{\alpha}:\Bbbk L_{Y}\to\Bbbk L_{X}}}$
defined on basis elements by $\rho_{Y,X}^{\alpha}(g)=g\bullet f_{Y,X}^{\alpha}$
is a non-zero homomorphism of $\Bbbk\OD_{n}$-modules.
\end{lem}

\begin{proof}
First note that if $g\in L_{Y}$ then by the fact that $\Lc$ is a
right congruence 
\[
g\,\Lc\,e_{Y}\implies gf_{Y,X}^{\alpha}\,\Lc\,e_{Y}f_{Y,X}^{\alpha}=f_{Y,X}^{\alpha}
\]
so 
\[
gf_{Y,X}^{\alpha}\in L_{X}.
\]
Therefore, a neater description of $\rho_{Y,X}^{\alpha}$ is $\rho_{Y,X}^{\alpha}(g)=gf_{Y,X}^{\alpha}$
and it is clearly a non-zero function. Since $L_{Y}$ and $L_{X}$
are in the same $\Jc$-class and $f_{Y,X}^{\alpha}\in L_{X}$, then
the fact that $\rho_{Y,X}^{\alpha}:\Bbbk L_{Y}\to\Bbbk L_{X}$ defined
by $\rho_{Y,X}^{\alpha}(g)=gf_{Y,X}^{\alpha}$ is a homomorphism of
$\Bbbk\OD_{n}$-modules follows from Green's lemma \cite[Lemma 2.2.1]{Howie1995},
and also stated explicitly in \cite[Lemma 3.5]{Stein2025}.
\end{proof}
\begin{cor}
\label{cor:Equal_size_dim_geq}If $X,Y\in P_{1}([n])$ with $|X|=|Y|$
then $\dim_{\Bbbk}\Hom(\Bbbk L_{Y},\Bbbk L_{X})\geq|\widetilde{H}_{Y,X}|$.
\end{cor}

\begin{proof}
If $|Y|=|X|=1$ then $\rho_{Y,X}^{{\bf 1}}\in\Hom(\Bbbk L_{Y},\Bbbk L_{X})$
is a non-zero homomorphism so 
\[
\dim_{\Bbbk}\Hom(\Bbbk L_{Y},\Bbbk L_{X})\geq1=|\widetilde{H}_{Y,X}|.
\]
If $|Y|=|X|>1$ then $\rho_{Y,X}^{{\bf 1}},\rho_{Y,X}^{{\bf -1}}\in\Hom(\Bbbk L_{Y},\Bbbk L_{X})$
are linearly independent homomorphisms. Indeed, 
\[
\alpha\rho_{Y,X}^{{\bf 1}}+\beta\rho_{Y,X}^{{\bf -1}}=0
\]
implies 
\[
\alpha\rho_{Y,X}^{{\bf 1}}(e_{Y})+\beta\rho_{Y,X}^{{\bf -1}}(e_{Y})=0
\]
so
\[
\alpha e_{Y}f_{Y,X}^{{\bf 1}}+\beta e_{Y}f_{Y,X}^{{\bf -1}}=0.
\]
Therefore, 
\[
\alpha f_{Y,X}^{{\bf 1}}+\beta f_{Y,X}^{{\bf -1}}=0
\]
so $\alpha=\beta=0$ as $f_{Y,X}^{{\bf 1}}$ and $f_{Y,X}^{{\bf -1}}$
are linearly independent elements of $\Bbbk L_{X}$. This establishes
$\dim_{\Bbbk}\Hom(\Bbbk L_{Y},\Bbbk L_{X})\geq2=|\widetilde{H}_{Y,X}|$.
\end{proof}
The more complicated case is when $X,Y\in P_{1}([n])$, $|X|=k$ and
$|Y|=k+1$. Set $Y=\{y_{1},\ldots,y_{k},y_{k+1}\}$ and $Y_{i}=Y\backslash\{y_{i}\}$.
Let $d_{Y,X}^{{\bf 1}}$ be the linear combination
\[
d_{Y,X}^{{\bf 1}}=\sum_{i=1}^{k+1}(-1)^{i+1}f_{Y_{i},X}^{{\bf 1}}
\]
and note that $d_{Y,X}^{{\bf 1}}\in\Bbbk L_{X}$.
\begin{lem}
\label{lem:Annihilate_d}Let $s\in\OD_{n}$. If $|\im(s)|\leq k$,
then $s\bullet d_{Y,X}^{{\bf 1}}=0$.
\end{lem}

\begin{proof}
We know that $\im(se_{Y})\subseteq\im(s)$ so $|\im(se_{Y})|\leq k$.
If $|\im(se_{Y})|<k$, then clearly $|\im(sf_{Y_{i},X}^{{\bf 1}})|=|\im(se_{Y}f_{Y_{i},X}^{{\bf 1}})|<k$
for every $i$. Hence, $sf_{Y_{i},X}^{{\bf 1}}\notin L_{X}$ and $s\bullet d_{Y,X}^{{\bf 1}}=0$
in this case. If |$\im(se_{Y})|=k$ then there exists a unique $j\in[k]$
such that $s(y_{j})=s(y_{j+1})$. In this case, $|\im(sf_{Y_{i},X})|<k$
for $i\notin\{j,j+1\}$, so $s\bullet f_{Y_{i},X}^{{\bf 1}}=0$ whenever
$i\notin\{j,j+1\}$. Finally, $sf_{Y_{j},X}=sf_{Y_{j+1},X}$ and they
have a different sign in $d_{Y,X}^{{\bf 1}}$ so we conclude that
$s\bullet d_{Y,X}^{{\bf 1}}=0$ as required.
\end{proof}
\begin{lem}
\label{lem:Delta_hom}The function $\delta_{Y,X}^{{\bf 1}}:\Bbbk L_{Y}\to\Bbbk L_{X}$
defined by $\delta_{Y,X}^{{\bf 1}}(g)=g\bullet d_{Y,X}^{{\bf 1}}$
is a non-zero homomorphism of $\Bbbk\OD_{n}$-modules.
\end{lem}

\begin{proof}
First note that $e_{Y}f_{Y_{i},X}^{{\bf 1}}=f_{Y_{i},X}^{{\bf 1}}$
for every $i\geq1$ since $Y_{i}\subseteq Y$. Therefore $\delta_{Y,X}^{{\bf 1}}(e_{Y})=d_{Y,X}^{{\bf 1}}$
so $\delta_{Y,X}^{{\bf 1}}$ is non-zero. Moreover, if $g\Lc e_{Y}$
then $gf_{Y_{i},X}^{{\bf 1}}\,\Lc\,e_{Y}f_{Y_{i},X}^{{\bf 1}}=f_{Y_{i},X}^{{\bf 1}}\in L_{X}$
so 
\[
gd_{Y,X}^{{\bf 1}}=\sum_{i=1}^{k+1}(-1)^{i+1}gf_{Y_{i},X}^{{\bf 1}}\in\Bbbk L_{X}
\]
for every $g\in L_{Y}$. Therefore, we can describe $\delta_{Y,X}^{{\bf 1}}$
by 
\[
\delta_{Y,X}^{{\bf 1}}(g)=gd_{Y,X}^{{\bf 1}}=\sum_{i=1}^{k+1}(-1)^{i+1}gf_{Y_{i},X}^{{\bf 1}}.
\]
It is left to show that $\delta_{Y,X}^{{\bf 1}}:\Bbbk L_{Y}\to\Bbbk L_{X}$
is a homomorphism of $\Bbbk\OD_{n}$-modules. Let $s\in\OD_{n}$,
$g\in L_{Y}$. We need to prove that 
\[
\delta_{Y,X}^{{\bf 1}}(s\bullet g)=s\bullet\delta_{Y,X}^{{\bf 1}}(g).
\]
If $sg\in L_{Y}$ then $\delta_{Y,X}^{{\bf 1}}(s\bullet g)=\delta_{Y,X}^{{\bf 1}}(sg)$.
Now, 
\[
s\bullet\delta_{Y,X}^{{\bf 1}}(g)=s\bullet gd_{Y,X}^{{\bf 1}}=s\bullet(g\bullet d_{Y,X}^{{\bf 1}})=sg\bullet d_{Y,X}^{{\bf 1}}=\delta_{Y,X}^{{\bf 1}}(sg)
\]
so $\delta_{Y,X}^{{\bf 1}}(s\bullet g)=s\bullet\delta_{Y,X}^{{\bf 1}}(g)$
in this case. The other case is when $sg\notin L_{Y}$ so $\delta_{Y,X}^{{\bf 1}}(s\bullet g)=\delta_{Y,X}^{{\bf 1}}(0)=0$.
We need to prove that $s\bullet\delta_{Y,X}^{{\bf 1}}(g)=s\bullet(gd_{Y,X}^{{\bf 1}})$
is zero as well. Note that $s\bullet(gd_{Y,X}^{{\bf 1}})=s\bullet(g\bullet d_{Y,X}^{{\bf 1}})=sg\bullet d_{Y,X}^{{\bf 1}}$,
and $|\im(sg)|\leq|\im(g)|=k+1$. If $|\im(sg)|=k+1$, then $sg\Jc g$,
and stability of finite monoids \cite[Theorem 1.13]{Steinberg2016}
implies $sg\in L_{Y}$, which is a contradiction. Therefore, $|\im(sg)|\leq k$
so $sg\bullet d_{Y,X}^{{\bf 1}}=0$ by \lemref{Annihilate_d}. This
finishes the proof.
\end{proof}
Let $d_{Y,X}^{{\bf -1}}$ be the linear combination
\[
d_{Y,X}^{{\bf -1}}=\sum_{i=1}^{k+1}(-1)^{i+1}f_{Y_{i},X}^{{\bf -1}},
\]
and note that $d_{Y,X}^{{\bf -1}}\in\Bbbk L_{X}$ and $d_{Y,X}^{{\bf 1}}f_{X,X}^{{\bf -1}}=d_{Y,X}^{{\bf -1}}$.
Since $\delta_{Y,X}^{{\bf 1}}:\Bbbk L_{Y}\to\Bbbk L_{X}$ and $\mbox{\ensuremath{\rho_{X,X}^{-{\bf 1}}:\Bbbk L_{X}\to\Bbbk L_{X}}}$
are both homomorphisms of $\Bbbk\OD_{n}$-modules we can define another
homomorphism $\mbox{\ensuremath{\delta_{Y,X}^{{\bf -1}}:\Bbbk L_{Y}\to\Bbbk L_{X}}}$
by $\delta_{Y,X}^{{\bf -1}}=\rho_{X,X}^{-{\bf 1}}\delta_{Y,X}^{{\bf 1}}$.
Note that for every $g\in L_{Y}$
\[
\delta_{Y,X}^{{\bf -1}}(g)=\rho_{X,X}^{-{\bf 1}}\delta_{Y,X}^{{\bf 1}}(g)=\rho_{X,X}^{-{\bf 1}}(gd_{Y,X}^{{\bf 1}})=gd_{Y,X}^{{\bf 1}}f_{X,X}^{{\bf -1}}=gd_{Y,X}^{{\bf -1}}.
\]

\begin{cor}
\label{cor:Non_equal_size_dim_eq}If $X,Y\in P_{1}([n])$ with $|Y|=|X|+1$
then $\dim_{\Bbbk}\Hom(\Bbbk L_{Y},\Bbbk L_{X})\geq|\widetilde{H}_{Y,X}|$.
\end{cor}

\begin{proof}
If $|X|=1$ then $\delta_{Y,X}^{{\bf 1}}\in\Hom(L_{Y},L_{X})$ is
a non-zero homomorphism so 
\[
\dim_{\Bbbk}\Hom(\Bbbk L_{Y},\Bbbk L_{X})\geq1=|\widetilde{H}_{Y,X}|.
\]
If $|X|=k>1$ then $\delta_{Y,X}^{{\bf 1}},\delta_{Y,X}^{{\bf -1}}\in\Hom(L_{Y},L_{X})$
are linearly independent homomorphisms. Indeed, 
\[
\alpha\delta_{Y,X}^{{\bf 1}}+\beta\delta_{Y,X}^{{\bf -1}}=0
\]
implies 
\[
\alpha\delta_{Y,X}^{{\bf 1}}(e_{Y})+\beta\delta_{Y,X}^{{\bf -1}}(e_{Y})=0
\]
so
\[
\alpha e_{Y}d_{Y,X}^{{\bf 1}}+\beta e_{Y}d_{Y,X}^{{\bf -1}}=0.
\]
Therefore, 
\[
\alpha d_{Y,X}^{{\bf 1}}+\beta d_{Y,X}^{{\bf -1}}=0
\]
which says 
\[
\alpha\sum_{i=1}^{k+1}(-1)^{i+1}f_{Y_{i},X}^{{\bf 1}}+\beta\sum_{i=1}^{k+1}(-1)^{i+1}f_{Y_{i},X}^{{\bf -1}}=0.
\]
As $\{f_{Y_{i},X}^{{\bf 1}},f_{Y_{i},X}^{{\bf -1}}\mid i=1,\ldots,k+1\}$
is a linearly independent set of $\Bbbk L_{X}$, this shows that $\alpha=\beta=0$
and establishes $\dim_{\Bbbk}\Hom(\Bbbk L_{Y},\Bbbk L_{X})\geq2=|\widetilde{H}_{Y,X}|$.
\end{proof}
\begin{prop}
\label{prop:Dim_of_hom}For every $X,Y\in P_{1}([n])$ we have $\dim_{\Bbbk}\Hom(\Bbbk L_{Y},\Bbbk L_{X})=|\widetilde{H}_{Y,X}|$.
\end{prop}

\begin{proof}
According to \corref{Equal_size_dim_geq} and \corref{Non_equal_size_dim_eq}
we know that $\dim_{\Bbbk}\Hom(\Bbbk L_{Y},\Bbbk L_{X})\geq|\widetilde{H}_{Y,X}|$
if $|Y|=|X|$ or $|Y|=|X|+1$. In all other cases, $|\widetilde{H}_{Y,X}|=0$,
so $\dim_{\Bbbk}\Hom(L_{Y},L_{X})\geq|\widetilde{H}_{Y,X}|$ holds
trivially. As noted above, we know that 
\[
\sum_{X,Y\in P_{1}([n])}|\widetilde{H}_{Y,X}|=|\OD_{n}|.
\]

Moreover, $\Bbbk\OD_{n}\simeq{\displaystyle \bigoplus_{X\in P_{1}([n])}}\Bbbk L_{X}$
and since $A^{\op}\simeq\Hom(A,A)$ for every algebra $A$ we obtain
an isomorphism of $\Bbbk$-vector spaces
\begin{align*}
\Bbbk\OD_{n}^{\op} & \simeq\Hom(\Bbbk\OD_{n},\Bbbk\OD_{n})\simeq\Hom\left({\displaystyle \bigoplus_{Y\in P_{1}([n])}}\Bbbk L_{Y},{\displaystyle \bigoplus_{X\in P_{1}([n])}}\Bbbk L_{X}\right)\\
 & \simeq\bigoplus_{X,Y\in P_{1}([n])}\Hom(\Bbbk L_{Y},\Bbbk L_{X}).
\end{align*}
Therefore, 
\[
|\OD_{n}|=\dim_{\Bbbk}\Bbbk\OD_{n}^{\op}=\sum_{X,Y\in P_{1}([n])}\dim_{\Bbbk}\Hom(\Bbbk L_{Y},\Bbbk L_{X})\geq\sum_{X,Y\in P_{1}([n])}|\widetilde{H}_{Y,X}|=|\OD_{n}|
\]
so it must be the case that 
\[
\dim_{\Bbbk}\Hom(\Bbbk L_{Y},\Bbbk L_{X})=|\widetilde{H}_{Y,X}|
\]
for every $X,Y\in P_{1}([n])$, as required.
\end{proof}
\begin{rem}
\propref{Dim_of_hom} is similar to other known results. Let $S$
be a finite semigroup and let $E\subseteq S$ be a subset of idempotents.
The set $E$ is partially ordered by the natural partial order on
idempotents: $e\leq f$ if and only if $ef=e=fe$. Assume that for
every $a\in S$ its set of right identities from $E$ has a minimum,
and its set of left identities from $E$ has a minimum. For $e,f\in E$
we set $\widetilde{H}_{e,f}$ to be the set of all elements $a\in S$
such that $e$ is their minimum left identity from $E$ and $f$ is
their minimum right identity from $E$ (if non-empty, this is an $\Ht$-class
\cite{Lawson1991}). We can define an equivalence relation $\Lt$
by $a\Lt b$ if $ae=a\iff be=b$ for every $e\in E$. If $e\in E$
and $\Lt(e)$ is the $\Lt$-class of $e$ then it is often the case
that $\Bbbk\Lt(e)$ is a projective $\Bbbk S$-module. In many cases
the dimension $\dim_{\Bbbk}\Hom(\Lt(e),\Lt(f))$ is precisely $|\widetilde{H}_{e,f}|$.
This holds for $\Jc$-trivial monoids \cite[Theorem 3.20]{Denton2011},
a special class of reduced $E$-Fountain semigroups \cite[Corollary 3.12]{Stein2025}
and $E$-semiadequate semigroups (yet unpublished). For $S=\OD_{n}$
we can choose $E=\{e_{X}\mid X\in P_{1}([n])\}$. Then, $\Lt=\Lc$
and $\widetilde{H}_{Y,X}=\widetilde{H}_{e_{Y},e_{X}}$ for $X,Y\in P_{1}([n])$.
Thus, $\OD_{n}$ provides another instance where the dimension of
homomorphisms between projective modules associated with $\Lt$-classes
is determined by the cardinality of $\Ht$-classes.
\end{rem}

Let $\mathcal{LC}_{n}$ denote the linear category whose objects are
the modules of the form $\Bbbk L_{X}$ for $X\in P_{1}([n])$, and
the morphisms are all the homomorphisms between them. As stated in
the preliminaries, the $\Bbbk$-algebra $\mathscr{A}[\mathcal{LC}_{n}]$
is isomorphic to $\Bbbk\OD_{n}^{\op}$, and therefore we can switch
our attention to finding a quiver presentation for the algebra of
$\mathcal{LC}_{n}^{\op}$. For convenience, we will initially continue
working with $\mathcal{LC}_{n}$ and reverse the morphisms at a later
stage.

\subsection{\label{subsec:The-skeletal-category}The skeletal category}

A category is called \emph{skeletal} if no two distinct objects are
isomorphic. Any category is equivalent to a skeletal category (which
is unique up to isomorphism) called its \emph{skeleton}. The skeleton
of a category $\mathcal{C}$ is the full subcategory consisting of
one object from every isomorphism class of $\mathcal{C}$. It is well-known
that algebras of equivalent categories are Morita equivalent \cite[Proposition 2.2]{Webb2007},
so the algebra of a category and the algebra of its skeleton have
the same quiver presentation. In this subsection we will describe
the skeleton of $\mathcal{LC}_{n}$, which we denote by $\mathcal{SL}_{n}$,
and in the rest of this section, we will seek a quiver presentation
for the algebra of $\mathcal{SL}_{n}$. To describe the skeleton,
we must determine which objects of $\mathcal{LC}_{n}$ are isomorphic.
This is not difficult with the information we already have.
\begin{lem}
The $\Bbbk\OD_{n}$-module $\Bbbk L_{X}$ is isomorphic to $\Bbbk L_{Y}$
if and only if $|X|=|Y|$.
\end{lem}

\begin{proof}
If $|Y|\neq|X|$ then we can assume without loss of generality that
$|Y|<|X|$. According to \propref{Dim_of_hom} $\dim_{\Bbbk}\Hom(\Bbbk L_{Y},\Bbbk L_{X})=|\widetilde{H}_{Y,X}|=0$
so $\Bbbk L_{Y}$ and $\Bbbk L_{X}$ are not isomorphic. If $|X|=|Y|$
then $\rho_{Y,X}^{{\bf 1}}\in\Hom(\Bbbk L_{Y},\Bbbk L_{X})$ and $\rho_{X,Y}^{{\bf 1}}\in\Hom(\Bbbk L_{X},\Bbbk L_{Y})$
are mutually inverse. Indeed, for every $g\in L_{Y}$ we have 
\[
\rho_{X,Y}^{{\bf 1}}(\rho_{Y,X}^{{\bf 1}}(g))=\rho_{Y,X}^{{\bf 1}}(g)f_{X,Y}^{{\bf 1}}=gf_{Y,X}^{{\bf 1}}f_{X,Y}^{{\bf 1}}=gf_{Y,Y}^{{\bf 1}}=g
\]
and likewise $\rho_{Y,X}^{{\bf 1}}(\rho_{X,Y}^{{\bf 1}}(g))=g$ for
$g\in L_{X}$. Therefore, $\Bbbk L_{X}\simeq\Bbbk L_{Y}$ as required. 
\end{proof}
For the skeleton $\mathcal{SL}_{n}$ we take one object from any isomorphism
class of objects. For every $\mbox{\ensuremath{k\in[n]}}$ we set
$L_{k}=L_{[k]}$ so we can take $\{\Bbbk L_{k}\mid1\leq k\leq n\}$
as the set of objects of the skeleton. For $\alpha\in\{{\bf 1},-{\bf 1}\}$
and $1\leq k\leq n$, set 
\[
f_{k}^{{\bf \alpha}}=f_{[k],[k]}^{\alpha},\quad\rho_{k}^{\alpha}=\rho_{[k],[k]}^{\alpha}
\]
so $\Hom(\Bbbk L_{k},\Bbbk L_{k})$ is spanned by $\{\rho_{k}^{{\bf 1}},\rho_{k}^{-{\bf 1}}\}$
according to \lemref{Rho_hom}. The homomorphism $\rho_{k}^{\alpha}$
is a representative of every $\rho_{Y,X}^{\alpha}$ with $|X|=|Y|=k$.
Note further that $\rho_{k}^{{\bf 1}}:\Bbbk L_{k}\to\Bbbk L_{k}$
is the identity function as $\rho_{k}^{{\bf 1}}(g)=gf_{k}^{{\bf {\bf 1}}}=g$
for every $g\in L_{k}$. For $\alpha\in\{{\bf 1},-{\bf 1}\}$, $1\leq k\leq n-1$
and $1\leq i\leq k+1$ we set 
\[
h_{k,i}^{\alpha}=f_{[k+1]\backslash\{i\},[k]}^{\alpha}
\]
and 
\[
d_{k}^{{\bf \alpha}}=\sum_{i=1}^{k+1}(-1)^{i+1}h_{k,i}^{\alpha}
\]
so according to \lemref{Delta_hom}, we have a homomorphism $\delta_{k}^{\alpha}:\Bbbk L_{k+1}\to\Bbbk L_{k}$
defined by $\mbox{\ensuremath{\delta_{k}^{\alpha}(g)=gd_{k}^{{\bf \alpha}}}}$.
The homomorphism $\text{\ensuremath{\delta_{k}^{\alpha}} is a representative of every \ensuremath{\delta_{Y,X}^{\alpha}} with \ensuremath{|X|=k}}$
and $|Y|=k+1$. By \propref{Dim_of_hom} we know that $\Hom(\Bbbk L_{k+1},\Bbbk L_{k})$
is spanned by $\{\delta_{k}^{{\bf 1}},\delta_{k}^{-{\bf 1}}\}$ and
$\Hom(\Bbbk L_{r},\Bbbk L_{k})=0$ if $r\notin\{k,k+1\}$. This gives
a concrete description of the skeleton $\mathcal{SL}_{n}$.

\subsection{\label{subsec:Composition-of-homomorphisms}Composition of homomorphisms}

In order to describe a quiver presentation we will have to calculate
the composition of each type of homomorphisms in the category $\mathcal{SL}_{n}$. 

First note that $\delta_{k}^{\alpha}\delta_{k+1}^{\beta}\in\Hom(\Bbbk L_{k+2},\Bbbk L_{k})=0$
so $\delta_{k}^{\alpha}\delta_{k+1}^{\beta}=0$ for $\alpha,\beta\in\{{\bf 1},-{\bf 1}\}$
and $1\leq k\leq n-2$. Moreover, $\rho_{k}^{{\bf 1}}$ is an identity
morphism so composition with $\rho_{k}^{{\bf 1}}$ is trivial. It
is left to consider compositions that involve $\rho_{k}^{{\bf -1}}.$
Since $f_{k}^{{\bf -{\bf 1}}}f_{k}^{{\bf -{\bf 1}}}=f_{k}^{{\bf {\bf 1}}}$
we obtain $\rho_{k}^{{\bf -1}}\rho_{k}^{{\bf -1}}(g)=gf_{k}^{{\bf -{\bf 1}}}f_{k}^{{\bf -{\bf 1}}}=gf_{k}^{{\bf 1}}=\rho_{k}^{{\bf 1}}(g)$
so $\rho_{k}^{{\bf -1}}\rho_{k}^{{\bf -1}}=\rho_{k}^{{\bf 1}}$. By
definition, $\delta_{k}^{-{\bf 1}}=\rho_{k}^{-{\bf 1}}\delta_{k}^{{\bf 1}}$
and therefore $\rho_{k}^{-{\bf 1}}\delta_{k}^{-{\bf 1}}=\rho_{k}^{-{\bf 1}}\rho_{k}^{-{\bf 1}}\delta_{k}^{{\bf 1}}=\delta_{k}^{{\bf 1}}$.
It is left to figure out what are the compositions $\delta_{k}^{{\bf 1}}\rho_{k+1}^{-{\bf 1}}$
and $\delta_{k}^{{\bf -1}}\rho_{k+1}^{-{\bf 1}}$.
\begin{lem}
The equality 
\[
\delta_{k}^{{\bf 1}}\rho_{k+1}^{-{\bf 1}}=(-1)^{k}\delta_{k}^{{\bf -1}}
\]
holds for every $1\leq k\leq n-1$.
\end{lem}

\begin{proof}
For every $x\in[k+1]$ we have 
\[
f_{k+1}^{{\bf {\bf -1}}}(x)=k+2-x
\]
and for every $x\in[k]$ and $1\leq i\leq k+1$ 
\[
h_{k,i}^{{\bf 1}}(x)=\begin{cases}
x & x<i\\
x+1 & i\leq x.
\end{cases}
\]
Therefore, for every $x\in[k]$
\begin{align*}
f_{k+1}^{{\bf -{\bf 1}}}h_{k,i}^{{\bf 1}}(x) & =\begin{cases}
k+2-x & x<i\\
k+1-x & i\leq x
\end{cases}=h_{k,k+2-i}^{{\bf -1}}(x).
\end{align*}
If $k<x\leq n$ then 
\[
f_{k+1}^{{\bf -{\bf 1}}}h_{k,i}^{{\bf 1}}(x)=f_{k+1}^{{\bf -{\bf 1}}}h_{k,i}^{{\bf 1}}(k)=f_{k+1}^{{\bf -{\bf 1}}}(k+1)=1=h_{k,k+2-i}^{{\bf -1}}(k)=h_{k,k+2-i}^{{\bf -1}}(x)
\]
so we establish 
\[
f_{k+1}^{{\bf -{\bf 1}}}h_{k,i}^{{\bf 1}}=h_{k,k+2-i}^{{\bf -1}}.
\]
Now,
\[
f_{k+1}^{{\bf -{\bf 1}}}d_{k}^{{\bf {\bf 1}}}=\sum_{i=1}^{k+1}(-1)^{i+1}f_{k+1}^{{\bf -{\bf 1}}}h_{k,i}^{{\bf 1}}=\sum_{i=1}^{k+1}(-1)^{i+1}h_{k,k+2-i}^{{\bf -1}}
\]
and rearranging the indices we obtain 
\[
\sum_{i=1}^{k+1}(-1)^{i+1}h_{k,k+2-i}^{{\bf -1}}=\sum_{i=1}^{k+1}(-1)^{k+3-i}h_{k,i}^{{\bf -1}}=(-1)^{k}\sum_{i=1}^{k+1}(-1)^{i+1}h_{k,i}^{{\bf -1}}=(-1)^{k}d_{k}^{-{\bf 1}}.
\]
Finally, 
\[
\delta_{k}^{{\bf 1}}\rho_{k+1}^{-{\bf 1}}(g)=\delta_{k}^{{\bf 1}}(gf_{k+1}^{{\bf -{\bf 1}}})=gf_{k+1}^{{\bf -{\bf 1}}}d_{k}^{{\bf 1}}=(-1)^{k}gd_{k}^{{\bf -1}}=(-1)^{k}\delta_{k}^{{\bf -1}}(g).
\]
\end{proof}
From $\delta_{k}^{{\bf 1}}\rho_{k+1}^{-{\bf 1}}=(-1)^{k}\delta_{k}^{{\bf -1}}$
we also obtain that 
\[
\delta_{k}^{{\bf -1}}\rho_{k+1}^{-{\bf 1}}=\rho_{k}^{-{\bf 1}}\delta_{k}^{{\bf 1}}\rho_{k+1}^{-{\bf 1}}=\rho_{k}^{-{\bf 1}}(-1)^{k}\delta_{k}^{{\bf -1}}=(-1)^{k}\delta_{k}^{{\bf 1}}.
\]

Before we summarize the relations, recall that $\Bbbk\OD_{n}$ is
Morita equivalent to $\mathscr{A}(\mathcal{SL}_{n}^{\op})$, so our
ultimate interest is in the opposite category $\mathcal{SL}_{n}^{\op}$.
It is convenient to switch to the opposite category at this point.
We will therefore list, in the next corollary, the opposites of the
relations we have obtained. In particular, note that $\delta_{k}^{\alpha}:\Bbbk L_{k}\to\Bbbk L_{k+1}$.
\begin{cor}
\label{cor:Relation_of_SL}Let $\alpha,\beta\in\{{\bf 1},-{\bf 1}\}$.
The following identities hold in the linear category $\mathcal{SL}_{n}^{\op}$:

\begin{align*}\delta_{k+1}^{\beta}\delta_{k}^{\alpha}&=0 &&\hspace{-1cm}(1\leq k\leq n-2)\\\rho_{k}^{\beta}\rho_{k}^{\alpha}&=\rho_{k}^{\beta\alpha}&&\hspace{-1cm}(1\leq k\leq n)\\\delta_{k}^{\beta}\rho_{k}^{\alpha}&=\delta_{k}^{\beta\alpha}&&\hspace{-1cm}(1\leq k\leq n-1)\\\rho_{k+1}^{{\bf 1}}\delta_{k}^{\alpha}&=\delta_{k}^{\alpha}&&\hspace{-1cm}(1\leq k\leq n-1)\\\rho_{k+1}^{-{\bf 1}}\delta_{k}^{\alpha}&=(-1)^{k}\delta_{k}^{-\alpha}&&\hspace{-1cm}(1\leq k\leq n-1).\end{align*}
\end{cor}

\subsection{\label{subsec:Change_of_bases}Change of basis}

In this subsection, we replace the basis of morphisms that we work
with by a more convenient choice. From now on, we assume that the
characteristic of $\Bbbk$ is different from $2$.

For $1\leq k\leq n$ set 
\[
\epsilon_{k}^{+}=\frac{\rho_{k}^{{\bf 1}}+\rho_{k}^{{\bf -1}}}{2},\quad\Delta_{k}^{+}=\frac{\delta_{k}^{{\bf 1}}+\delta_{k}^{{\bf -1}}}{2}
\]
 and for $2\leq k\leq n$ set 
\[
\epsilon_{k}^{-}=\frac{\rho_{k}^{{\bf 1}}-\rho_{k}^{{\bf -1}}}{2},\quad\Delta_{k}^{-}=\frac{\delta_{k}^{{\bf 1}}-\delta_{k}^{{\bf -1}}}{2}.
\]

Clearly, $\Hom(\Bbbk L_{k},\Bbbk L_{k})$ is spanned by $\{\epsilon_{k}^{+},\epsilon_{k}^{-}\}$
and $\Hom(\Bbbk L_{k},\Bbbk L_{k+1})$ is spanned by $\{\Delta_{k}^{{\bf +}},\Delta_{k}^{-}\}$.
The description of the relations between these morphisms will be more
convenient than the one in \corref{Relation_of_SL}.
\begin{lem}
\label{lem:New_Relations} Let $\alpha,\beta\in\{+,-\}$. The following
identities hold in the linear category $\mathcal{SL}_{n}^{\op}$:

\begin{align*} \Delta_{k+1}^{\beta}\Delta_{k}^{\alpha} & =0  && \hspace{-1cm} (1\leq k\leq n-2)\\ \epsilon_{k}^{\alpha}\epsilon_{k}^{\alpha} & =\epsilon_{k}^{\alpha} && \hspace{-1cm}(1\leq k\leq n) \\ \epsilon_{k}^{\alpha}\epsilon_{k}^{-\alpha} & =0 && \hspace{-1cm}(1\leq k\leq n) \\ \Delta_{k}^{\alpha}\epsilon_{k}^{\alpha} & =\Delta_{k}^{\alpha} && \hspace{-1cm} (1\leq k\leq n-1) \\ \Delta_{k}^{\alpha}\epsilon_{k}^{-\alpha} & =0 && \hspace{-1cm} (1\leq k\leq n-1) \\\epsilon_{k+1}^{\alpha}\Delta_{k}^{\alpha} & =\begin{cases} \Delta_{k}^{\alpha} & k\:\:\text{even}\\ 0 & k\:\:\text{odd} \end{cases} && \hspace{-1cm} (1\leq k\leq n-1) \\ \epsilon_{k+1}^{-\alpha}\Delta_{k}^{\alpha} & =\begin{cases} 0 & k\:\:\text{even}\\ \Delta_{k}^{\alpha} & k\:\:\text{odd} \end{cases}  && \hspace{-1cm} (1\leq k\leq n-1) \end{align*}
\end{lem}

\begin{proof}
All of the identities are routine to verify. We will do only the last.
If $k$ is even then 
\begin{align*}
\epsilon_{k+1}^{-}\Delta_{k}^{+} & =\frac{\rho_{k+1}^{{\bf 1}}-\rho_{k+1}^{{\bf -1}}}{2}\cdot\frac{\delta_{k}^{{\bf 1}}+\delta_{k}^{{\bf -1}}}{2}=0\\
\epsilon_{k+1}^{+}\Delta_{k}^{-} & =\frac{\rho_{k+1}^{{\bf 1}}+\rho_{k+1}^{{\bf -1}}}{2}\cdot\frac{\delta_{k}^{{\bf 1}}-\delta_{k}^{{\bf -1}}}{2}=0
\end{align*}
and if $k$ is odd $\rho_{k+1}^{-{\bf 1}}\delta_{k}^{\alpha}=-\delta_{k}^{-\alpha}$
so
\begin{align*}
\epsilon_{k+1}^{-}\Delta_{k}^{+} & =\frac{\rho_{k+1}^{{\bf 1}}-\rho_{k+1}^{{\bf -1}}}{2}\cdot\frac{\delta_{k}^{{\bf 1}}+\delta_{k}^{{\bf -1}}}{2}=\frac{\delta_{k}^{{\bf 1}}+\delta_{k}^{{\bf -1}}}{4}+\frac{\delta_{k}^{{\bf -1}}+\delta_{k}^{{\bf 1}}}{4}=\Delta_{k}^{+}\\
\epsilon_{k+1}^{+}\Delta_{k}^{-} & =\frac{\rho_{k+1}^{{\bf 1}}+\rho_{k+1}^{{\bf -1}}}{2}\cdot\frac{\delta_{k}^{{\bf 1}}-\delta_{k}^{{\bf -1}}}{2}=\frac{\delta_{k}^{{\bf 1}}-\delta_{k}^{{\bf -1}}}{4}+\frac{-\delta_{k}^{{\bf -1}}+\delta_{k}^{{\bf 1}}}{4}=\Delta_{k}^{-}.
\end{align*}
\end{proof}
Set $\mathcal{A}_{n}=\mathcal{\mathscr{A}}(\mathcal{SL^{\op}}_{n})$.
In the next subsection we will find a quiver presentation for $\text{\ensuremath{\mathcal{A}_{n}}}$
(which is also a quiver presentation for $\Bbbk\OD_{n}$). We conclude
this section by observing that morphisms of the form $\epsilon_{k}^{\alpha}$
form a complete set of primitive orthogonal idempotents for $\text{\ensuremath{\mathcal{A}_{n}}}$.
\begin{lem}
\label{lem:Primitive_idpts}The set $\{\epsilon_{k}^{+}\}_{k=1}^{n}\cup\{\epsilon_{k}^{-}\}_{k=2}^{n}$
is a complete set of primitive orthogonal idempotents for $\mathcal{A}_{n}$.
\end{lem}

\begin{proof}
Since $\epsilon_{k}^{+}+\epsilon_{k}^{-}=\rho_{k}^{{\bf 1}}$ for
every $k\in\{2,\ldots,n\}$ and $\epsilon_{1}^{+}=\rho_{1}^{{\bf 1}}$,
it is clear that 
\[
\sum_{k=1}^{n}\epsilon_{k}^{+}+\sum_{k=2}^{n}\epsilon_{k}^{-}=\sum_{k=1}^{n}\rho_{k}^{{\bf 1}}.
\]
The sum of all identity morphisms is the identity of the algebra $\mathcal{A}_{n}$.
Moreover, since $\epsilon_{k}^{\alpha}$ is an endomorphisms of $\Bbbk L_{k}$
it is clear that $\epsilon_{k}^{\beta}\epsilon_{r}^{\alpha}=0$ if
$r\neq k$ (for $\alpha,\beta\in\{+,-\}$). According to \lemref{New_Relations}
it is clear that 
\[
(\epsilon_{k}^{+})^{2}=\epsilon_{k}^{+},\quad(\epsilon_{k}^{-})^{2}=\epsilon_{k}^{-},\quad\epsilon_{k}^{+}\epsilon_{k}^{-}=\epsilon_{k}^{-}\epsilon_{k}^{+}=0
\]
so they are indeed orthogonal idempotents, Finally, it is clear that
$\epsilon_{k}^{\alpha}\mathcal{A}_{n}\epsilon_{k}^{\alpha}$ is spanned
by $\{\epsilon_{k}^{\alpha}\}$ so it is of dimension $1$. Therefore,
every idempotent of the form $\epsilon_{k}^{\alpha}$ is primitive
in $\mathcal{A}_{n}$.
\end{proof}

\subsection{\label{subsec:A-quiver-presentation}A quiver presentation}

In this subsection, we use the relations from \lemref{New_Relations}
to obtain a quiver presentation of $\mathcal{A}_{n}$. We begin by
describing its radical.
\begin{lem}
\label{lem:The_Radical}The radical of $\mathcal{\mathcal{A}}_{n}$
is $R=\Span\{\Delta_{k}^{\alpha}\mid k\in[n-1],\quad\alpha\in\{{\bf +},-\}\}$.
Moreover, $\text{\mbox{\ensuremath{\Rad^{2}\mathcal{\mathcal{A}}_{n}=0}}}$.
\end{lem}

\begin{proof}
Since $\Delta_{k+1}^{\beta}\Delta_{k}^{\alpha}=0$ it is clear that
$R^{2}=0$. It is left to show that it is the radical. Let $G_{n}$
be the subcategory of $\mathcal{SL}_{n}$ with the same set of objects,
but whose morphisms are only the endomorphisms. In $G_{n}$, all morphisms
are linear combinations of morphisms of the form $\epsilon_{k}^{{\bf \alpha}}$.
For $2\leq k\leq n$, the set $\{\epsilon_{k}^{+},\epsilon_{k}^{-}\}$
is a basis for the endomorphisms of $\Bbbk L_{k}$. Being orthogonal
idempotents, they generate an algebra isomorphic to $\Bbbk^{2}$.
For $\Bbbk L_{1}$, the endomorphism algebra is isomorphic to $\Bbbk$.
Therefore, $\mathscr{A}(G_{n})\simeq\Bbbk\times{\displaystyle \prod_{k=2}^{n}}\Bbbk^{2}\simeq\Bbbk^{2n-1}$,
and hence it is a semisimple algebra. Since every morphism of $\mathcal{SL}_{n}$
lies either in $G_{n}$ or in $R$, it follows that $\mathcal{A}_{n}/R\simeq\mathscr{A}(G_{n})$.
Therefore, $R$ is a nilpotent ideal with the property that $\mathcal{A}_{n}/R$
is semisimple. This shows that $R=\Rad\mathcal{A}_{n}.$ 
\end{proof}
\begin{cor}
The algebra $\mathcal{A}_{n}$ is split basic.
\end{cor}

\begin{proof}
Follows from $\mathcal{A}_{n}/\Rad\mathcal{A}_{n}\simeq\Bbbk^{2n-1}$.
\end{proof}
Since $\mathcal{A}_{n}$ is a split basic algebra, it admits a quiver
presentation. Let $Q$ be the quiver of the algebra $\mathcal{A}_{n}$.
The vertices of $Q$ are in one-to-one correspondence with the primitive
idempotents. If $e_{1},e_{2}$ are two primitive idempotents, then
the number of arrows from $e_{1}$ to $e_{2}$ is 
\[
\dim_{\Bbbk}(e_{2}+\Rad^{2}\mathcal{A}_{n})(\Rad\mathcal{A}_{n}/\Rad^{2}\mathcal{A}_{n})(e_{1}+\Rad^{2}\mathcal{A}_{n}).
\]
 We have already seen that $\Rad^{2}\mathcal{A}_{n}=0$ so the arrows
from $e_{1}$ to $e_{2}$ correspond to a basis of $e_{2}\Rad\mathcal{A}_{n}e_{1}$.
The radical $\Rad\mathcal{A}_{n}$ is spanned by morphisms from $\Bbbk L_{k}$
to $\Bbbk L_{k+1}$, so $\epsilon_{r}^{\beta}\Rad\mathcal{A}_{n+1}\epsilon_{k}^{\alpha}=0$
if $r\neq k+1$. Therefore, there are no arrows from $\epsilon_{k}^{\alpha}$
to $\epsilon_{r}^{\beta}$ if $r\neq k+1$. It is left to handle the
$r=k+1$ case.
\begin{lem}
\label{lem:Odd_arrows}If $k$ is even then there is one arrow from
$\epsilon_{k}^{+}$ to $\epsilon_{k+1}^{+}$ and one arrow from $\epsilon_{k}^{-}$
to $\epsilon_{k+1}^{-}$.
\end{lem}

\begin{proof}
According to \lemref{The_Radical}, $\epsilon_{k+1}^{\beta}\Rad\mathcal{A}_{n}\epsilon_{k}^{\alpha}$
is spanned by $\epsilon_{k+1}^{\beta}\Delta_{k}^{+}\epsilon_{k}^{\alpha}$
and $\epsilon_{k+1}^{\beta}\Delta_{k}^{-}\epsilon_{k}^{\alpha}.$
Using the relations from \lemref{New_Relations}, it is routine to
check that
\begin{align*}
\epsilon_{k+1}^{+}\Delta_{k}^{+}\epsilon_{k}^{+} & =\Delta_{k}^{+},\quad\epsilon_{k+1}^{-}\Delta_{k}^{-}\epsilon_{k}^{-}=\Delta_{k}^{-}\\
\epsilon_{k+1}^{+}\Delta_{k}^{+}\epsilon_{k}^{-} & =\epsilon_{k+1}^{-}\Delta_{k}^{+}\epsilon_{k}^{+}=\epsilon_{k+1}^{-}\Delta_{k}^{+}\epsilon_{k}^{-}=0\\
\epsilon_{k+1}^{-}\Delta_{k}^{-}\epsilon_{k}^{+} & =\epsilon_{k+1}^{+}\Delta_{k}^{-}\epsilon_{k}^{-}=\epsilon_{k+1}^{+}\Delta_{k}^{-}\epsilon_{k}^{+}=0.
\end{align*}
Therefore, there is one arrow from $\epsilon_{k}^{+}$ to $\epsilon_{k+1}^{+}$
that we associate with $\Delta_{k}^{+}\in\mathcal{A}_{n+1}$, and
there is one arrow from $\epsilon_{k}^{-}$ to $\epsilon_{k+1}^{-}$
that we associate with $\Delta_{k}^{-}\in\mathcal{A}_{n+1}$. There
are no arrows from $\epsilon_{k}^{+}$ to $\epsilon_{k+1}^{-}$ or
from $\epsilon_{k}^{-}$ to $\epsilon_{k+1}^{+}$.
\end{proof}
\begin{lem}
\label{lem:Even_arrows}If $k$ is odd then there is one arrow from
$\epsilon_{k}^{+}$ to $\epsilon_{k+1}^{-}$ and one arrow from $\epsilon_{k}^{-}$
to $\epsilon_{k+1}^{+}$.
\end{lem}

\begin{proof}
According to \lemref{The_Radical}, $\epsilon_{k+1}^{\beta}\Rad\mathcal{A}\epsilon_{k}^{\alpha}$
is spanned by $\epsilon_{k+1}^{\beta}\Delta_{k}^{+}\epsilon_{k}^{\alpha}$
and $\epsilon_{k+1}^{\beta}\Delta_{k}^{-}\epsilon_{k}^{\alpha}.$
Using the relations from \lemref{New_Relations}, it is routine to
check that
\begin{align*}
\epsilon_{k+1}^{-}\Delta_{k}^{+}\epsilon_{k}^{+} & =\Delta_{k}^{+},\quad\epsilon_{k+1}^{+}\Delta_{k}^{-}\epsilon_{k}^{-}=\Delta_{k}^{-}\\
\epsilon_{k+1}^{+}\Delta_{k}^{+}\epsilon_{k}^{+} & =\epsilon_{k+1}^{+}\Delta_{k}^{+}\epsilon_{k}^{-}=\epsilon_{k+1}^{-}\Delta_{k}^{+}\epsilon_{k}^{-}=0\\
\epsilon_{k+1}^{-}\Delta_{k}^{-}\epsilon_{k}^{-} & =\epsilon_{k+1}^{-}\Delta_{k}^{-}\epsilon_{k}^{+}=\epsilon_{k+1}^{+}\Delta_{k}^{-}\epsilon_{k}^{+}=0.
\end{align*}
Therefore, there is one arrow from $\epsilon_{k}^{+}$ to $\epsilon_{k+1}^{-}$
that we associate with $\Delta_{k}^{+}\in\mathcal{A}_{n+1}$, and
there is one arrow from $\epsilon_{k}^{-}$ to $\epsilon_{k+1}^{+}$
that we associate with $\Delta_{k}^{-}\in\mathcal{A}_{n+1}$. There
are no arrows from $\epsilon_{k}^{-}$ to $\epsilon_{k+1}^{-}$ or
from $\epsilon_{k}^{+}$ to $\epsilon_{k+1}^{+}$.
\end{proof}
The arrows of the quiver are associated with the morphisms $\Delta_{k}^{+}$
and $\Delta_{k}^{-}$. Then, in view of the relation $\Delta_{k+1}^{\beta}\Delta_{k}^{\alpha}=0$,
it is clear that composition of two arrows equals zero in the quiver
presentation.

We can now finally conclude:
\begin{thm}
Let $\Bbbk$ be a field whose characteristic is not $2$. The quiver
of the algebra $\Bbbk\OD_{n}$ has $2n-1$ vertices: $\epsilon_{k}^{+}$
for $1\leq k\leq n$ and $\epsilon_{k}^{-}$ for $2\leq k\leq n$.
For even $k$ there is one arrow from $\epsilon_{k}^{+}$ to $\epsilon_{k+1}^{+}$
and one arrow from $\epsilon_{k}^{-}$ to $\epsilon_{k+1}^{-}$. For
odd $k$, there is one arrow from $\epsilon_{k}^{+}$ to $\epsilon_{k+1}^{-}$
and one arrow from $\epsilon_{k}^{-}$ to $\epsilon_{k+1}^{+}$. In
the presentation, every composition of consecutive arrows is zero.
\end{thm}

\begin{example}
The quiver of $\Bbbk\OD_{6}$ is displayed in the following figure:
\end{example}

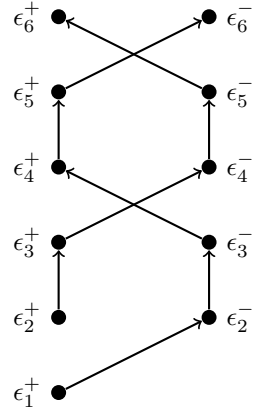
\begin{center} \begin{tikzpicture}     \node[circle, fill, inner sep=2pt, label=left:$\epsilon^+_1$] (p1) at (0,0) {};     \foreach \k in {2,3,4,5,6}         \node[circle, fill, inner sep=2pt, label=left:$\epsilon^+_{\k}$] (p\k) at (0,\k-1) {};     \foreach \k in {2,3,4,5,6}         \node[circle, fill, inner sep=2pt, label=right:$\epsilon^-_{\k}$] (m\k) at (2,\k-1) {};          \draw[->, line width=0.8pt] (p1) -- (m2);     \foreach \k in {2,4}         \draw[->, line width=0.8pt] (p\k) -- (p\the\numexpr\k+1);     \foreach \k in {2,4}         \draw[->, line width=0.8pt] (m\k) -- (m\the\numexpr\k+1);     \foreach \k in {3,5}         \draw[->, line width=0.8pt] (p\k) -- (m\the\numexpr\k+1);     \foreach \k in {3,5}         \draw[->, line width=0.8pt] (m\k) -- (p\the\numexpr\k+1); \end{tikzpicture} \end{center}

Of course, if we forget the names of the vertices we obtain two straightline
paths.
\begin{cor}
Let $\Bbbk$ be a field whose characteristic is not $2$. The quiver
of the algebra $\Bbbk\OD_{n}$ has two straightline paths, one with
$n-1$ vertices and one with $n$ vertices. In the presentation, every
composition of consecutive arrows is zero.
\end{cor}

\begin{rem}
The quiver has two connected components. Therefore, there exists a
non-trivial central idempotent $e\in\Bbbk\OD_{n}$ such that $\Bbbk\OD_{n}$
decomposes into a direct product of algebras
\[
\Bbbk\OD_{n}\simeq e\Bbbk\OD_{n}e\times(1-e)\Bbbk\OD_{n}(1-e).
\]
We leave open the question of describing this idempotent as a linear
combination of elements of $\OD_{n}$.
\end{rem}

\section{\label{sec:COD_section}A covering semidirect product}

Order-preserving functions $f\in\Op_{n}$ are in one-to-one correspondence
with pairs $X\in P_{1}([n])$, $\mbox{\ensuremath{Y\in P([n])}}$
such that $|Y|=|X|$. This is also the case for order-reversing functions
$f\in\OD_{n}$, which gives the impression that $\OD_{n}$ is ``twice''
the monoid $\Op_{n}$. However, this is not true even in terms of
cardinality, since $|\OD_{n}|=2|\Op_{n}|-n.$ The reason is that there
is no difference between constant order-preserving functions and constant
order-reversing functions. In this section we introduce a new monoid
that we dub $\COD_{n}$ (stands for ``Covering $\OD_{n}")$. In $\COD_{n}$
we artificially distinguish between $f_{Y,X}^{{\bf 1}}$ and $f_{Y,X}^{-{\bf 1}}$
even if they are constant functions. The monoid $\OD_{n}$ is a natural
quotient of $\COD_{n}$, and in several aspects, $\COD_{n}$ is indeed
``twice'' $\Op_{n}$. We prove that $\COD_{n}\simeq\Op_{n}\rtimes\mathbb{Z}_{2}$
(so in particular, $|\COD_{n}|=2|\Op_{n}|$). Moreover, imitating
the approach of \secref{Order_preserving}, we show that $\Bbbk\COD_{n}\simeq\Bbbk\Op_{n}\times\Bbbk\Op_{n}$
for every field $\Bbbk$ whose characteristic is not $2$.

\subsection{Definition and basic observations}

For $z\in[n]$, we denote by $c_{z}$ the constant function $c_{z}\in\OD_{n}$
defined by $c_{z}(x)=z$ for every $x\in[n]$. Note that $c_{z}=f_{\{z\},\{1\}}^{{\bf 1}}=f_{\{z\},\{1\}}^{{\bf -1}}$.
Consider the direct product $\OD_{n}\times\mathbb{Z}_{2}$ where $\mathbb{Z}_{2}$
is the multiplicative group $\{{\bf 1},{\bf -1}\}$. We denote by
$\COD_{n}$ the following submonoid 
\[
\COD_{n}=\{(f_{Y,X}^{\alpha},\alpha)\mid f_{Y,X}^{\alpha}\in\OD_{n}\text{ is not a constant function}\}\cup\{(c_{z},{\bf 1}),(c_{z},-{\bf 1})\mid z\in[n]\}.
\]
 
\begin{lem}
The set $\COD_{n}$ is indeed a submonoid of $\OD_{n}\times\mathbb{Z}_{2}$.
\end{lem}

\begin{proof}
Let $f_{Y_{1},X_{1}}^{\alpha_{1}},f_{Y_{2},X_{2}}^{\alpha_{2}}\in\OD_{n}$
such that $f_{Y_{2},X_{2}}^{\alpha_{2}}f_{Y_{1},X_{1}}^{\alpha_{1}}=f_{Z,W}^{\beta}$.
If $f_{Z,W}^{\beta}$ is not a constant function, then it must be
the case that $\beta=\alpha_{2}\alpha_{1}$, so 
\[
(f_{Y_{2},X_{2}}^{\alpha_{2}},\alpha_{2})(f_{Y_{1},X_{1}}^{\alpha_{1}},\alpha_{1})=(f_{Z,W}^{\beta},\alpha_{2}\alpha_{1})=(f_{Z,W}^{\beta},\beta)\in\COD_{n}.
\]
If \textbf{$f_{Z,W}^{\beta}=c_{z}$ }is a constant function then 
\[
(f_{Y_{2},X_{2}}^{\alpha_{2}},\alpha_{2})(f_{Y_{1},X_{1}}^{\alpha_{1}},\alpha_{1})=(f_{Z,W}^{\beta},\alpha_{2}\alpha_{1})=(c_{z},\alpha_{2}\alpha_{1})\in\COD_{n}.
\]
\end{proof}
We can think of $\COD_{n}$ as essentially the same monoid as $\OD_{n}$,
but we keep an artificial distinction between $f_{Y,X}^{{\bf 1}}$
and $f_{Y,X}^{-{\bf 1}}$ even when they are constant functions. Note
that this adds $n$ ``artificial'' constant functions so $|\COD_{n}|=|\OD_{n}|+n$.
Since $|\OD_{n}|=2|\Op_{n}|-n$ we have that $\mbox{\ensuremath{|\COD_{n}|=2|\Op_{n}|}}$.

The definition of $\COD_{n}$ might seem artificial, but from algebraic
point of view it has some interesting properties. It is easy to see
that $\OD_{n}$ is a homomorphic image of $\COD_{n}$. We can define
an equivalence relation $\sim$ on $\COD_{n}$ by the rule that $(f,\alpha)\sim(g,\beta)\iff f=g$.
Clearly, $\sim$ is a congruence and the quotient monoid is isomorphic
to $\OD_{n}$. Next, we show that $\COD_{n}$ is a semidirect product.
Let $\id=f_{[n],[n]}^{{\bf 1}}$ be the identity element of $\OD_{n}$
and set $\h=f_{[n],[n]}^{-{\bf 1}}$. The group $\mathbb{Z}_{2}\simeq\{\id,\h\}$
acts on $\Op_{n}$ by conjugation, for every $g\in\{\id,\h\}$ we
set 
\[
g\ast f=gfg
\]
(note that $g=g^{-1}$). We can define a semidirect product $\Op_{n}\rtimes\mathbb{Z}_{2}$.
For every $f_{1},f_{2}\in\Op_{n}$ and $g_{1},g_{2}\in\mathbb{Z}_{2}$,
composition in the semidirect product is defined by 
\[
(f_{2},g_{2})\cdot(f_{1},g_{1})=(f_{2}(g_{2}\ast f_{1}),g_{2}g_{1}).
\]

It will be convenient to set $\sgn(\id)={\bf 1}$ and $\sgn(\h)={\bf -1}$.
Clearly, $\sgn:\{\id,\h\}\to\{{\bf 1},{\bf -1}\}$ is an isomorphism.
\begin{lem}
\label{lem:Semidirect_product_iso}The function $\varphi:\Op_{n}\rtimes\mathbb{Z}_{2}\to\COD_{n}$
defined by 
\[
\varphi((f_{Y,X}^{{\bf 1}},g))=(f_{Y,X}^{{\bf 1}}g,\sgn(g))
\]
is an isomorphism of monoids.
\end{lem}

\begin{proof}
First note that if $g=\id$ then $(f_{Y,X}^{{\bf 1}}g,\sgn(g))=(f_{Y,X}^{{\bf 1}},{\bf 1})\in\COD_{n}$.
If $g=\h$ then $f_{Y,X}^{{\bf 1}}g=f_{Z,W}^{{\bf -1}}$ for some
$Z,W$, so 
\[
(f_{Y,X}^{{\bf 1}}g,\sgn(g))=(f_{Z,W}^{{\bf -1}},{\bf -1})\in\COD_{n},
\]
hence $\varphi$ is well-defined. Next, we show that $\varphi$ is
a homomorphism. Indeed,
\begin{align*}
\varphi((f_{Y_{2},X_{2}}^{{\bf 1}},g_{2})(f_{Y_{1},X_{1}}^{{\bf 1}},g_{1})) & =\varphi((f_{Y_{2},X_{2}}^{{\bf 1}}(g_{2}\ast f_{Y_{1},X_{1}}^{{\bf 1}}),g_{2}g_{1}))\\
 & =\varphi((f_{Y_{2},X_{2}}^{{\bf 1}}g_{2}f_{Y_{1},X_{1}}^{{\bf 1}}g_{2},g_{2}g_{1}))\\
 & =(f_{Y_{2},X_{2}}^{{\bf 1}}g_{2}f_{Y_{1},X_{1}}^{{\bf 1}}g_{2}g_{2}g_{1},\sgn(g_{2}g_{1}))\\
 & =(f_{Y_{2},X_{2}}^{{\bf 1}}g_{2}f_{Y_{1},X_{1}}^{{\bf 1}}g_{1},\sgn(g_{2}g_{1}))\\
 & =(f_{Y_{2},X_{2}}^{{\bf 1}}g_{2},\sgn(g_{2}))\cdot(f_{Y_{1},X_{1}}^{{\bf 1}}g_{1},\sgn(g_{1}))\\
 & =\varphi((f_{Y_{2},X_{2}}^{{\bf 1}},g_{2}))\varphi((f_{Y_{1},X_{1}}^{{\bf 1}},g_{1})).
\end{align*}

Finally, we show that $\varphi$ is injective. If $\varphi((f_{1},g_{1}))=\varphi((f_{2},g_{2}))$
then $(f_{1}g_{1},\sgn(g_{1}))=(f_{2}g_{2},\sgn(g_{2}))$. Now, $\sgn(g_{1})=\sgn(g_{2})$
implies $g_{1}=g_{2}$ and therefore $f_{1}g_{1}=f_{2}g_{1}$. Composing
with $g_{1}$ on the right we obtain $f_{1}=f_{2}$ as required. This
shows that $\varphi$ is injective. Since $|\Op_{n}\rtimes\mathbb{Z}_{2}|=|\COD_{n}|=2|\Op_{n}|$,
we have established that $\varphi$ is an isomorphism.
\end{proof}
\begin{rem}
We remark that $\psi:\Op_{n}\rtimes\mathbb{Z}_{2}\to\OD_{n}$ defined
by $\psi((f,g))=fg$ is the natural epimorphism corresponding to the
quotient $\COD_{n}/\sim$ discussed above. Therefore, 
\[
\psi((f_{2},g_{2}))=\psi((f_{1},g_{1}))
\]
holds if and only if $(f_{2},g_{2})=(f_{1},g_{1})$ or $f_{1}=c_{z}=f_{2}$
where $c_{z}\in\Op_{n}$ is a constant function. 
\end{rem}

From now on, when considering the monoid $\COD_{n}$, we will use
the abbreviation $f_{Y,X}^{\alpha}$ for the pair $(f_{Y,X}^{\alpha},\alpha)$.
Therefore, we treat $f_{Y,X}^{{\bf 1}}$ and $f_{Y,X}^{{\bf -1}}$
as distinct elements of $\COD_{n}$, even though they represent the
same element of $\OD_{n}$ in the case that $f_{Y,X}^{{\bf 1}}$ is
a constant function. 

Clearly, Green's relations in $\COD_{n}$ are easily derived from
Green's relations of $\OD_{n}$. If $f_{Y_{2},X_{2}}^{\beta},f_{Y_{1},X_{1}}^{\alpha}\in\COD_{n}$
then:

\begin{align*}
f_{Y_{1},X_{1}}^{\alpha}\,\Rc\,f_{Y_{2},X_{2}}^{\beta} & \iff Y_{1}=Y_{2}\\
f_{Y_{1},X_{1}}^{\alpha}\,\Lc\,f_{Y_{2},X_{2}}^{\beta} & \iff X_{1}=X_{2}\\
f_{Y_{1},X_{1}}^{\alpha}\,\Jc\,f_{Y_{2},X_{2}}^{\beta} & \iff|X_{1}|=|Y_{1}|=|X_{2}|=|Y_{2}|\\
f_{Y_{1},X_{1}}^{\alpha}\,\Hc\,f_{Y_{2},X_{2}}^{\beta} & \iff X_{1}=X_{2},\enspace Y_{1}=Y_{2}.
\end{align*}
Note that in $\COD_{n}$, each $\Hc$-class has $2$ elements, and
every maximal subgroup is isomorphic to $\mathbb{Z}_{2}$.

\subsection{The algebra $\Bbbk\COD_{n}$}

In this subsection, we study the algebra $\Bbbk\COD_{n}$ where $\Bbbk$
is a field whose characteristic is not $2$. We adopt the approach
taken in \secref{Order_preserving}. Since the proofs are similar
to those already presented, we either omit them or sketch them.
\begin{lem}
The sandwich matrix $P$ of every $\Jc$-class $J_{k}$ of $\COD_{n}$
is right invertible. Therefore, for every field $\Bbbk$ and every
$X\in P_{1}([n])$ the $\Bbbk\COD_{n}$-module $\Bbbk L_{X}$ is a
projective module and 
\[
\Bbbk\COD_{n}\simeq\bigoplus_{X\in P_{1}([n])}\Bbbk L_{X}
\]
is an isomorphism of $\Bbbk\COD_{n}$-modules. 
\end{lem}

\begin{proof}
Similar to \lemref{Sandwitch_right_invertible} and \propref{Peirce_decomposition}.
\end{proof}
The description of homomorphisms between induced left Schützenberger
modules is similar to the one we found in \subsecref{Construction-of-homomorphisms}.
\begin{prop}
Let $X,Y\in P_{1}([n])$.
\begin{enumerate}
\item If $|Y|=|X|$ and $\alpha\in\{{\bf 1},{\bf -1}\}$, then the function
$\mbox{\ensuremath{\rho_{Y,X}^{\alpha}:\Bbbk L_{Y}\to\Bbbk L_{X}}}$
defined on basis elements by $\rho_{Y,X}^{\alpha}(g)=g\bullet f_{Y,X}^{\alpha}$
is a non-zero homomorphism of $\Bbbk\COD_{n}$-modules. Moreover,
$\{\rho_{Y,X}^{{\bf 1}},\rho_{Y,X}^{{\bf -1}}\}$ forms a basis for
the $\Bbbk$-vector space $\Hom(\Bbbk L_{Y},\Bbbk L_{X})$.
\item If $|Y|=|X|+1$ and $\alpha\in\{{\bf 1},{\bf -1}\}$, then the function
$\delta_{Y,X}^{{\bf \alpha}}:\Bbbk L_{Y}\to\Bbbk L_{X}$ defined by
$\delta_{Y,X}^{{\bf \alpha}}(g)=g\bullet d_{Y,X}^{{\bf \alpha}}$
where 
\[
d_{Y,X}^{{\bf \alpha}}={\displaystyle \sum_{i=1}^{k}}(-1)^{i+1}f_{Y_{i},X}^{{\bf \alpha}}
\]
 is a non-zero homomorphism of $\Bbbk\COD_{n}$-modules. Moreover,
$\{\delta_{Y,X}^{{\bf 1}},\delta_{Y,X}^{{\bf -1}}\}$ forms a basis
for the $\Bbbk$-vector space $\Hom(\Bbbk L_{Y},\Bbbk L_{X})$.
\item If $|Y|\notin\{|X|,|X|+1\}$ then $\Hom(\Bbbk L_{Y},\Bbbk L_{X})=0$.
\end{enumerate}
\end{prop}

Let $\mathcal{CLC}_{n}$ be the linear category whose objects are
the modules of the form $\Bbbk L_{X}$ for $X\in P_{1}([n])$, and
the morphisms are all homomorphisms between them. So the only difference
between $\mathcal{CLC}_{n}$ and $\mathcal{LC}_{n}$ is that $\mathcal{CLC}_{n}$
has an extra morphism $\rho_{Y,X}^{{\bf -1}}:\Bbbk L_{Y}\to\Bbbk L_{X}$
if $|Y|=|X|=1$ (that is, $Y=X=\{1\}$) and an extra morphism $\delta_{Y,X}^{{\bf {\bf -1}}}:\Bbbk L_{Y}\to\Bbbk L_{X}$
if $|Y|=2$, $|X|=1$.

We know that $\Bbbk\COD_{n}$ is isomorphic to $\mathscr{A}[\mathcal{CLC}_{n}^{\op}]$.
Here we prefer not to switch to the skeleton of the category. Nevertheless,
it is clear that the description of composition in the skeleton obtained
in \subsecref{Composition-of-homomorphisms} also describes composition
in the entire category. The morphism $\rho_{k}^{\alpha}$ is a representative
of $\rho_{Y,X}^{\alpha}$ where $X,Y\in P_{1}([n])$ with $|Y|=|X|=k$,
and $\delta_{k}^{\alpha}$ is a representative of $\delta_{Y,X}^{\alpha}$
where $X,Y\in P_{1}([n])$ with $|Y|=k+1$, $|X|=k$. Therefore, from
\corref{Relation_of_SL} we can deduce:
\begin{cor}
\label{cor:Relation_of_SCL}Let $\alpha,\beta\in\{{\bf 1},-{\bf 1}\}$.
The following identities hold in the linear category $\mathcal{CLC}_{n}^{\op}$:

\begin{align*}\delta_{Z,Y}^{\beta}\delta_{Y,X}^{\alpha}&=0&&\hspace{-1cm}(|Z|=|Y|+1=|X|+2)\\\rho_{Z,Y}^{\beta}\rho_{Y,X}^{\alpha}&=\rho_{Z,X}^{\beta\alpha}&&\hspace{-1cm}(|Z|=|Y|=|X|)\\\delta_{Z,Y}^{\beta}\rho_{Y,X}^{\alpha}&=\delta_{Z,X}^{\beta\alpha}&&\hspace{-1cm}(|Z|=|Y|+1=|X|+1)\\\rho_{Z,Y}^{{\bf 1}}\delta_{Y,X}^{\alpha}&=\delta_{Z,X}^{\alpha}&&\hspace{-1cm}(|Z|=|Y|=|X|+1)\\\rho_{Z,Y}^{-{\bf 1}}\delta_{Y,X}^{\alpha}&=(-1)^{|X|}\delta_{Z.X}^{-\alpha}&&\hspace{-1cm}(|Z|=|Y|=|X|+1)\end{align*}
\end{cor}

As in \subsecref{Change_of_bases}, we can define 
\[
\epsilon_{Y,X}^{+}=\frac{\rho_{Y,X}^{{\bf 1}}+\rho_{Y,X}^{{\bf -1}}}{2},\quad\epsilon_{Y,X}^{-}=\frac{\rho_{Y,X}^{{\bf 1}}-\rho_{Y,X}^{{\bf -1}}}{2}
\]
\[
\Delta_{Y,X}^{^{+}}=\frac{\delta_{Y,X}^{{\bf 1}}+\delta_{Y,X}^{{\bf -1}}}{2},\quad\Delta_{Y,X}^{^{-}}=\frac{\delta_{Y,X}^{{\bf 1}}-\delta_{Y,X}^{{\bf -1}}}{2}.
\]
As in \lemref{New_Relations} we obtain:
\begin{lem}
\label{lem:New_Relations_CLC} Let $\alpha,\beta\in\{+,-\}$. The
following identities hold in the linear category $\mathcal{CLC}_{n}^{\op}$:

\begin{align*}\Delta_{Z,Y}^{\beta}\Delta_{Y,X}^{\alpha}&=0 &&\hspace{-1cm}(|Z|=|Y|+1=|X|+2)\\\epsilon_{Z,Y}^{\alpha}\epsilon_{Y,X}^{\alpha}&=\epsilon_{Z,X}^{\alpha}&&\hspace{-1cm}(|Z|=|Y|=|X|)\\\epsilon_{Z,Y}^{\alpha}\epsilon_{Y,X}^{-\alpha}&=0&&\hspace{-1cm}(|Z|=|Y|=|X|)\\\Delta_{Z,Y}^{\alpha}\epsilon_{Y,X}^{\alpha}&=\Delta_{Z,X}^{\alpha}&&\hspace{-1cm}(|Z|=|Y|+1=|X|+1)\\\Delta_{Z,Y}^{\alpha}\epsilon_{Y,X}^{-\alpha}&=0&&\hspace{-1cm}(|Z|=|Y|+1=|X|+1)\\\epsilon_{Z,Y}^{\alpha}\Delta_{Y,X}^{\alpha}&=\begin{cases}\Delta_{Z,X}^{\alpha} & |X|\:\:\text{even}\\0 & |X|\:\:\text{odd}\end{cases}&&\hspace{-1cm}(|Z|=|Y|=|X|+1)\\\epsilon_{Z,Y}^{-\alpha}\Delta_{Y,X}^{\alpha}&=\begin{cases}0 & |X|\:\:\text{even}\\\Delta_{Z,X}^{\alpha} & |X|\:\:\text{odd}\end{cases}&&\hspace{-1cm}(|Z|=|Y|=|X|+1)\end{align*}
\end{lem}

Set $\mathcal{CA}_{n}=\mathcal{\mathscr{A}}(\mathcal{CLC}_{n}^{\op})$
and note that $\mathcal{CA}_{n}\simeq\Bbbk\COD_{n}$. As in \lemref{Primitive_idpts},
the set 
\[
E=\{\epsilon_{X,X}^{+},\epsilon_{X,X}^{-}\mid X\in P_{1}([n])\}
\]
is a complete set of primitive orthogonal idempotents for $\mathcal{CA}_{n}$.

\subsection{A decomposition into a direct product}

Let $\mathcal{D}_{n}$ be the category defined as follows. Its objects
are in one-to-one correspondence with sets $X\in P_{1}([n])$. For
every $X,Y\in P_{1}([n])$ with $|Y|=|X|$ there is a morphisms $r_{Y,X}$
from $X$ to $Y$. For every $X,Y\in P_{1}([n])$ with $|Y|=|X|+1$
there is a morphisms $d_{Y,X}$ from $X$ to $Y$. Moreover, for every
two objects $X,Y\in P_{1}([n])$, there is a zero morphism $z_{Y,X}$.
Composition is defined as follows:
\begin{align*}
r_{Z,Y}r_{Y,X} & =r_{Z,X}\quad(|Z|=|Y|=|X|)\\
r_{Z,Y}d_{Y,X} & =d_{Z,X}\quad(|Z|=|Y|=|X|+1)\\
d_{Z,Y}r_{Y,X} & =d_{Z,X}\quad(|Z|=|Y|+1=|X|+1)\\
d_{Z,Y}d_{Y,X} & =z_{Z,X}\quad(|Z|=|Y|+1=|X|+2)
\end{align*}
In addition, the zero morphisms have the expected absorbing properties.
It is proved in \cite[Theorem 5.7]{Stein2024preprint} that the contracted
category algebra $\Bbbk_{0}\mathcal{D}_{n}$ is isomorphic to the
monoid algebra $\Bbbk\Op_{n}$. We will show now that it is also closely
related to $\Bbbk\COD_{n}$ through the algebra $\mathcal{CA}_{n}$.
\begin{defn}
We define a function $F:\Bbbk_{0}\mathcal{D}_{n}\to\mathcal{CA}_{n}$
on basis elements as follows: 

\begin{align*}
F(r_{Y,X}) & =\begin{cases}
\epsilon_{Y,X}^{+} & |X|=0\text{ or }1\pmod 4\\
\epsilon_{Y,X}^{-} & |X|=2\text{ or }3\pmod 4
\end{cases}\\
F(d_{Y,X}) & =\begin{cases}
\Delta_{Y,X}^{+} & |X|=0\text{ or }1\pmod 4\\
\Delta_{Y,X}^{-} & |X|=2\text{ or }3\pmod 4
\end{cases}
\end{align*}
\end{defn}

\begin{lem}
\label{lem:AlternatingFunctionIsHom}The function $F$ is a homomorphism
of $\Bbbk$-algebras.
\end{lem}

\begin{proof}
We need to show that $F$ is multiplicative. Let $m,m^{\prime}$ be
two morphisms in $\mathcal{D}_{n}$. There are four cases:
\begin{align*}
m^{\prime} & =d_{Z,W},\,m=d_{Y,X}\quad\quad m^{\prime}=r_{Z,W},\,m=r_{Y,X}\\
m^{\prime} & =d_{Z,W},\,m=r_{Y,X}\quad\quad m^{\prime}=r_{Z,W},\,m=d_{Y,X}
\end{align*}
 If $W\neq Y$ then clearly $m^{\prime}m=0$ in $\Bbbk_{0}\mathcal{D}_{n}$
and $F(m^{\prime})F(m)=0$ in $\mathcal{CA}_{n}$. Therefore, it is
left to handle these four cases when $W=Y$. First note that 
\[
F(d_{Z,Y})F(d_{Y,X})=\Delta_{Z,Y}^{\beta}\Delta_{Y,X}^{\alpha}=0=F(0)=F(d_{Z,Y}d_{Y,X})
\]
since $d_{Z,Y}d_{Y,X}=0$ in the contracted algebra $\Bbbk_{0}\mathcal{D}_{n}$.
Next,
\begin{align*}
F(r_{Z,Y})F(r_{Y,X}) & =\begin{cases}
\epsilon_{Z,Y}^{+}\epsilon_{Y,X}^{+} & |X|=0\text{ or }1\pmod 4\\
\epsilon_{Z,Y}^{-}\epsilon_{Y,X}^{-} & |X|=2\text{ or }3\pmod 4
\end{cases}\\
 & =\begin{cases}
\epsilon_{Z,X}^{+} & |X|=0\text{ or }1\pmod 4\\
\epsilon_{Z,X}^{-} & |X|=2\text{ or }3\pmod 4
\end{cases}=F(r_{Z,X})=F(r_{Z,Y}r_{Y,X})
\end{align*}
and 
\begin{align*}
F(d_{Z,Y})F(r_{Y,X}) & =\begin{cases}
\Delta_{Z,Y}^{+}\epsilon_{Y,X}^{+} & |X|=0\text{ or }1\pmod 4\\
\Delta_{Z,Y}^{-}\epsilon_{Y,X}^{-} & |X|=2\text{ or }3\pmod 4
\end{cases}\\
 & =\begin{cases}
\Delta_{Z,X}^{+} & |X|=0\text{ or }1\pmod 4\\
\Delta_{Z,X}^{-} & |X|=2\text{ or }3\pmod 4
\end{cases}=F(d_{Z,X})=F(d_{Z,Y}r_{Y,X}).
\end{align*}
Finally, we consider $F(r_{Z,Y})F(d_{Y,X})$ where $|Z|=|Y|=|X|+1$.
We check all the options: 
\begin{itemize}
\item If $|X|=0\pmod 4$ then 
\[
F(r_{Z,Y})F(d_{Y,X})=\epsilon_{Z,Y}^{+}\Delta_{Y,X}^{+}=\Delta_{Z,X}^{+}=F(d_{Z,X})=F(r_{Z,Y}d_{Y,X}).
\]
\item If $|X|=1\pmod 4$ then 
\[
F(r_{Z,Y})F(d_{Y,X})=\epsilon_{Z,Y}^{-}\Delta_{Y,X}^{+}=\Delta_{Z,X}^{+}=F(d_{Z,X})=F(r_{Z,Y}d_{Y,X}).
\]
\item If $|X|=2\pmod 4$ then 
\[
F(r_{Z,Y})F(d_{Y,X})=\epsilon_{Z,Y}^{-}\Delta_{Y,X}^{-}=\Delta_{Z,X}^{-}=F(d_{Z,X})=F(r_{Z,Y}d_{Y,X}).
\]
\item If $|X|=3\pmod 4$ then 
\[
F(r_{Z,Y})F(d_{Y,X})=\epsilon_{Z,Y}^{+}\Delta_{Y,X}^{-}=\Delta_{Z,X}^{-}=F(d_{Z,X})=F(r_{Z,Y}d_{Y,X}).
\]
\end{itemize}
This finishes the proof.
\end{proof}
\begin{defn}
Define another function $F^{\prime}:\Bbbk_{0}\mathcal{D}_{n}\to\mathcal{CA}_{n}$
on basis elements as follows: 
\begin{align*}
F^{\prime}(r_{Y,X}) & =\begin{cases}
\epsilon_{Y,X}^{-} & |X|=0\text{ or }1\pmod 4\\
\epsilon_{Y,X}^{+} & |X|=2\text{ or }3\pmod 4
\end{cases}\\
F^{\prime}(d_{Y,X}) & =\begin{cases}
\Delta_{Y,X}^{-} & |X|=0\text{ or }1\pmod 4\\
\Delta_{Y,X}^{+} & |X|=2\text{ or }3\pmod 4
\end{cases}
\end{align*}
\end{defn}

The function $F^{\prime}$ is also a homomorphism of $\Bbbk$-algebras,
and the proof is similar to that of \lemref{AlternatingFunctionIsHom}.
We now show that the images of $F$ and $F^{\prime}$ are in two distinct
disconnected parts of the algebra.
\begin{lem}
\label{lem:Disconnected_Cat}For every morphisms $m,m^{\prime}$ in
$\mathcal{D}_{n}$, we have 
\[
F(m)F^{\prime}(m^{\prime})=F^{\prime}(m^{\prime})F(m)=0.
\]
\end{lem}

\begin{proof}
We show $F^{\prime}(m^{\prime})F(m)=0$ since the proof of the other
statement is similar. There are four cases:
\begin{align*}
m^{\prime} & =d_{Z,W},\,m=d_{Y,X}\quad\quad m^{\prime}=r_{Z,W},\,m=r_{Y,X}\\
m^{\prime} & =d_{Z,W},\,m=r_{Y,X}\quad\quad m^{\prime}=r_{Z,W},\,m=d_{Y,X}
\end{align*}
 If $W\neq Y$ then clearly 
\[
F^{\prime}(m^{\prime})F(m)=0
\]
so it is left to handle these four cases when $W=Y$. For every value
of $|X|$, 
\[
F^{\prime}(d_{Z,Y})F(d_{Y,X})=\Delta_{Z,Y}^{\beta}\Delta_{Y,X}^{\alpha}=0
\]
for some $\alpha,\beta\in\{+,-\}$. Next, 
\[
F^{\prime}(r_{Z,Y})F(r_{Y,X})=\begin{cases}
\epsilon_{Z,Y}^{-}\epsilon_{Y,X}^{+} & |X|=0\text{ or }1\pmod 4\\
\epsilon_{Z,Y}^{+}\epsilon_{Y,X}^{-} & |X|=2\text{ or }3\pmod 4
\end{cases}=0.
\]
Finally, 
\[
F^{\prime}(d_{Z,Y})F(r_{Y,X})=F^{\prime}(d_{Z,Y}r_{Y,Y})F(r_{Y,X})=F^{\prime}(d_{Z,Y})F^{\prime}(r_{Y,Y})F(r_{Y,X})=0
\]
and likewise
\[
F^{\prime}(r_{Z,Y})F(d_{Y,X})=F^{\prime}(r_{Z,Y})F(r_{Y,Y}d_{Y,X})=F^{\prime}(r_{Z,Y})F(r_{Y,Y})F(d_{Y,X})=0.
\]
\end{proof}
We can now conclude.
\begin{prop}
There is an isomorphism of $\Bbbk$-algebras $\Bbbk_{0}\mathcal{D}_{n}\times\Bbbk_{0}\mathcal{D}_{n}\simeq\mathcal{CA}_{n}$.
\end{prop}

\begin{proof}
Define $G:\Bbbk_{0}\mathcal{D}_{n}\times\Bbbk_{0}\mathcal{D}_{n}\to\mathcal{CA}_{n}$
by $G(m,m^{\prime})=F(m)+F(m^{\prime})$. According to \lemref{Disconnected_Cat},
\begin{align*}
G(m_{2},m_{2}^{\prime})G(m_{1},m_{1}^{\prime}) & =(F(m_{2})+F^{\prime}(m_{2}^{\prime}))(F(m_{1})+F^{\prime}(m_{1}^{\prime}))\\
 & =F(m_{2})F(m_{1})+F^{\prime}(m_{2}^{\prime})F^{\prime}(m_{1}^{\prime})\\
 & =F(m_{2}m_{1})+F^{\prime}(m_{2}^{\prime}m_{1}^{\prime})\\
 & =G(m_{2}m_{1},m_{2}^{\prime}m_{1}^{\prime})=G((m_{2},m_{2}^{\prime})(m_{1},m_{1}^{\prime}))
\end{align*}

so $G$ is a homomorphism. The image of $G$ contains all morphisms
of the form $\Delta_{Y,X}^{\alpha}$ and $\epsilon_{Y,X}^{\alpha}$,
which span the algebra $\mathcal{CA}_{n}$, so $G$ is onto. Since
\[
\dim_{\Bbbk}\left(\Bbbk_{0}\mathcal{D}_{n}\times\Bbbk_{0}\mathcal{D}_{n}\right)=2|\Op_{n}|=\dim_{\Bbbk}\mathcal{CA}_{n}
\]
we deduce that $G$ is an isomorphism. 
\end{proof}
Since $\mathcal{CA}_{n}\simeq\Bbbk\COD_{n}$ and $\Bbbk_{0}\mathcal{D}_{n}\simeq\Bbbk\Op_{n}$
we have the following theorem.
\begin{thm}
\label{thm:AlgebraCodIsomorphis}Let $\Bbbk$ be a field whose characteristic
is not $2$. There is an isomorphism of algebras $\Bbbk\COD_{n}\simeq\Bbbk\Op_{n}\times\Bbbk\Op_{n}$.
\end{thm}

\begin{rem}
We know that $\COD_{n}\simeq\Op_{n}\rtimes\mathbb{Z}_{2}$ by \lemref{Semidirect_product_iso}.
On the other hand, 
\[
\Bbbk\Op_{n}\times\Bbbk\Op_{n}=\Bbbk\Op_{n}\otimes\Bbbk^{2},
\]
and since $\Bbbk^{2}\simeq\Bbbk\mathbb{Z}_{2}$ we have
\[
\Bbbk\Op_{n}\otimes\Bbbk^{2}\simeq\Bbbk\Op_{n}\otimes\Bbbk\mathbb{Z}_{2}\simeq\Bbbk[\Op_{n}\times\mathbb{Z}_{2}].
\]
Therefore, \thmref{AlgebraCodIsomorphis} can also be interpreted
as saying that 
\[
\Bbbk[\Op_{n}\rtimes\mathbb{Z}_{2}]\simeq\Bbbk[\Op_{n}\times\mathbb{Z}_{2}].
\]
\end{rem}

As a final observation, we mention that the quiver presentation of
$\Bbbk\Op_{n}$ (for every field $\Bbbk$) is known to be a straight
line path with $n$ vertices, such that all compositions of consecutive
arrows are equal to $0$. This is a consequence of the Dold-Kan correspondence
\cite{Dold1958,Kan1958} (also stated more explicitly in \cite[Remark 5.8]{Stein2024preprint}).
From \thmref{AlgebraCodIsomorphis} we deduce:
\begin{cor}
Let $\Bbbk$ be a field whose characteristic is not $2$. The quiver
of the algebra $\Bbbk\COD_{n}$ has two straightline paths with $n$
vertices. In the presentation, every composition of consecutive arrows
is zero.
\end{cor}

\textbf{AI Disclosure}: GPT-4 was used to assist with some of the
English wording and phrasing. All mathematical content is the original
work of the author.

\bibliographystyle{plain}
\bibliography{library}

\end{document}